\documentclass[psamsfonts,reqno]{amsart}
\usepackage{amssymb,eucal}
\usepackage[usenames]{color}

\usepackage{amscd}

\usepackage{graphics}

\usepackage{mathrsfs}

\theoremstyle{plain}

\newtheorem{lemma}{Lemma}[section]
\newtheorem{theorem}[lemma]{Theorem}

\newtheorem{corollary}[lemma]{Corollary}

\newtheorem*{sclaim}{Claim}

\newtheorem*{stat}{\name}
\newcommand{\name}{testing}

\theoremstyle{definition}
\newtheorem{definition}[lemma]{Definition}

\theoremstyle{remark}
\newtheorem{remark}[lemma]{Remark}

\newtheorem*{remark*}{Remark}

\newenvironment{all}[1]{\renewcommand{\name}{#1}\begin{stat}}
                         {\end{stat}}

\newcommand{\qedc}{{\qed}~{\rm Claim~{\theclaim}.}}
\newcommand{\qedsc}{{\qed}~{\rm Claim.}}

\newenvironment{scproof}
{\begin{proof}[Proof of Claim.]}
{\qedsc\renewcommand{\qed}{}\end{proof}}

\numberwithin{equation}{section}

\newcommand{\pup}[1]{\textup{(}{#1}\textup{)}}

\newcommand{\set}[1]{\{#1\}}
\newcommand{\setm}[2]{\set{#1\mid#2}}
\newcommand{\Set}[1]{\left\{#1\right\}}
\newcommand{\Setm}[2]{\Set{#1\mid#2}}

\newcommand{\famm}[2]{(#1\mid#2)}

\newcommand{\Pow}{\mathfrak{P}}

\newcommand{\dnw}{\mathbin{\downarrow}}
\newcommand{\upw}{\mathbin{\uparrow}}

\newcommand{\Qone}{\mathbf{(Q1)}}
\newcommand{\Qtwo}{\mathbf{(Q2)}}

\newcommand{\cC}{\mathcal{C}}
\newcommand{\cD}{\mathcal{D}}

\newcommand{\cG}{\mathcal{G}}

\newcommand{\cI}{\mathcal{I}}

\newcommand{\cK}{\mathcal{K}}
\newcommand{\cL}{\mathcal{L}}
\newcommand{\cM}{\mathcal{M}}
\newcommand{\cN}{\mathcal{N}}

\newcommand{\cV}{\mathcal{V}}
\newcommand{\cW}{\mathcal{W}}

\newcommand{\zero}{\mathbf{0}}
\newcommand{\one}{\mathbf{1}}
\newcommand{\two}{\mathbf{2}}

\newcommand{\ba}{\boldsymbol{a}}
\newcommand{\bb}{\boldsymbol{b}}

\newcommand{\by}{\boldsymbol{y}}

\newcommand{\bff}{\boldsymbol{f}}
\newcommand{\bgg}{\boldsymbol{g}}

\newcommand{\bA}{\boldsymbol{A}}
\newcommand{\bB}{\boldsymbol{B}}

\newcommand{\bU}{\boldsymbol{U}}

\newcommand{\CG}{\boldsymbol{C}}
\newcommand{\GA}{\boldsymbol{G}}

\newcommand{\id}{\mathrm{id}}
\newcommand{\jz}{$(\vee,0)$}

\newcommand{\jzs}{\jz-semi\-lat\-tice}

\newcommand{\jzh}{\jz-ho\-mo\-mor\-phism}

\newcommand{\res}{\mathbin{\restriction}}

\DeclareMathOperator{\card}{card}

\DeclareMathOperator{\crita}{crit}
\newcommand{\crit}[2]{\crita({{#1};{#2}})}

\DeclareMathOperator{\Def}{Def}
\DeclareMathOperator{\Sub}{Sub}
\DeclareMathOperator{\At}{At}

\DeclareMathOperator{\rng}{rng}
\DeclareMathOperator{\Max}{Max}
\DeclareMathOperator{\Con}{Con}
\newcommand{\Conc}{\Con_{\mathrm{c}}}

\DeclareMathOperator{\lh}{lh}
\DeclareMathOperator{\Var}{{\bf{Var}}}

\newcommand{\IC}{\cI}
\newcommand{\JC}{\cI_1}
\newcommand{\KC}{\cI_2}

\newcommand{\tosurj}{\mathbin{\twoheadrightarrow}}

\newcommand{\xF}{\mathbf{F}}

\newcommand{\kposet}[4]{{{#1}} \mathbin{\boxtimes}_{{#3}} {{#4}} }

\newcommand{\dual}[1]{{#1}^{\mathrm{d}}}

\begin{document}

\title[Critical points]{The possible values of critical points between varieties of lattices}

\author[P.~Gillibert]{Pierre Gillibert}

\address{LMNO, CNRS UMR 6139\\
D\'epartement de Math\'ematiques, BP 5186\\
Universit\'e de Caen, Campus 2\\
14032 Caen cedex\\
France}
\email{pgillibert@yahoo.fr}
\urladdr{http://www.math.unicaen.fr/\~{}giliberp/}

% 06B20 : Varieties of lattices
% 08A30 : Subalgebras, congruence relations
% 08A55 : Partial algebras
% 08B25 : Products, amalgamated products, and other kinds of limits and colimits [See also 18A30]
% 08B26 : Subdirect products and subdirect irreducibility

\thanks{This work was partially supported by the institutional grant MSM 0021620839}

\begin{abstract}
We denote by $\Conc L$ the \jzs\ of all finitely generated congruences of a lattice~$L$. For varieties (i.e., equational classes) $\cV$ and~$\cW$ of lattices such that~$\cV$ is contained neither in~$\cW$ nor in its dual, and such that every simple member of~$\cW$ contains a prime interval, we prove that there exists a bounded lattice $A\in\cV$ with at most~$\aleph_2$ elements such that $\Conc A$ is not isomorphic to $\Conc B$ for any $B\in\cW$. The bound~$\aleph_2$ is optimal. As a corollary of our results, there are continuum many congruence classes of locally finite varieties of (bounded) modular lattices.
\end{abstract}
\maketitle

\section{Introduction}\label{S:Intro}
\subsection{Background}
An \emph{algebra} (in the sense of universal algebra) is a nonempty set~$A$ endowed with a collection of maps (``operations'') from finite powers of~$A$ to~$A$. One of the most fundamental invariants associated with an algebra~$A$ is the lattice~$\Con A$ of all \emph{congruences} of~$A$, that is, the equivalence relations on~$A$ compatible with all the operations of~$A$. It is often more convenient to work with the \jzs~$\Conc A$ of all \emph{finitely generated congruences} of~$A$ (for precise definitions we refer the reader to Section~\ref{S:Basic}). We set
 \[
 \Conc\cV=\setm{S}{(\exists A\in\cV)(S\cong\Conc A)},\quad\text{for every class }\cV
 \text{ of algebras},
 \]
and we call $\Conc\cV$ the \emph{compact congruence class of~$\cV$}. This object has been especially studied for~$\cV$ a \emph{variety} (or \emph{equational class}) \emph{of algebras}, that is, the class of all algebras that satisfy a given set of \emph{identities} (in a given similarity type).

For a variety~$\cV$, the following fundamental questions arise:

\begin{itemize}
\item[$\Qone$] Is~$\Conc\cV$ determined by a reasonably small fragment of itself---for example, is it determined by its \emph{finite} members?

\item[$\Qtwo$] In what extent does~$\Conc\cV$ characterize~$\cV$?
\end{itemize}

While both questions remain largely mysterious in full generality, partial answers are known in a number of situations. Here is an example that illustrates the difficulties underlying Question~$\Qone$ above. We are working with \emph{varieties of lattices}. Denote by~$\cD$ the variety of all distributive lattices and by~$\cM_3$ the variety of lattices generated by the five-element modular nondistributive lattice~$M_3$ (cf. Figure~\ref{F:treillis_M3_N5}, page~\pageref{F:treillis_M3_N5}). The finite members of both~$\Conc\cD$ and~$\Conc\cM_3$ are exactly the finite Boolean lattices. It is harder to prove that there exists a \emph{countable} lattice~$K\in\cM_3$ such that $\Conc K\not\cong\Conc D$ for any $D\in\cD$. This can be obtained easily from the results of Plo\v s\v cica in~\cite{Ploscica04}, but a direct construction is also possible, as follows: denote by~$S$ a copy of the two-atom Boolean lattice in~$M_3$; let~$K$ be the lattice of all eventually constant sequences of elements of~$M_3$ such that the limit belongs to~$S$. (This construction is a precursor of the \emph{condensates} introduced in~\cite{G1}.)

This suggests a way to \emph{measure} the ``containment defect'' of~$\Conc\cV$ into~$\Conc\cW$, for varieties (not necessarily in the same similarity type)~$\cV$ and~$\cW$. The \emph{critical point} between~$\cV$ and~$\cW$ is defined as
 \begin{align*}
 \crit{\cV}{\cW}=\begin{cases}
 \min\setm{\card S}{S\in(\Conc\cV)-(\Conc\cW)},&
 \text{if }\Conc\cV\not\subseteq\Conc\cW,\\
 \infty,&\text{if }\Conc\cV\subseteq\Conc\cW.
 \end{cases}
 \end{align*}
The example above shows that
 \[
 \crit{\cM_3}{\cD}=\aleph_0.
 \]
Using varieties generated by finite non-modular lattices it is easy to construct \emph{finite} critical points.
By using techniques from \emph{infinite combinatorics}, introduced in~\cite{NonMeas}, Plo\v s\v cica found in \cite{Ploscica00, Ploscica03} varieties of bounded lattices with critical point~$\aleph_2$; the bounds of those examples were subsequently removed, by the author of the present paper, in~\cite{G2}. For example, if we denote by~$\cM_n$ the variety generated by the lattice of length~$2$ with~$n$ atoms, then
 \[
 \crit{\cM_m}{\cM_n}=\aleph_2\quad\text{if }m>n\ge 3.
 \]
In~\cite{G1} we find a pair of finitely generated modular lattice varieties with critical point~$\aleph_1$, thus answering a 2002 question of T\r{u}ma and Wehrung from~\cite{CLPSurv}.

In ``crossover'' contexts (varieties on different similarity types), the situation between lattices and congruence-permutable varieties (say, \emph{groups} or \emph{modules}) is quite instructive. The five-element lattice~$M_3$ is isomorphic to the congruence lattice of a group (take the Klein group) but not to the congruence lattice of any lattice (for it fails distributivity), and it is easily seen to be the smallest such example. Thus, if we denote by~$\cL$ the variety of all lattices and by~$\cG$ the variety of all groups, then
 \[
 \crit{\cG}{\cL}=5.
 \]
Much harder techniques, also originating from~\cite{NonMeas}, are applied in~\cite{RTW}, yielding the following result:
 \[
 \crit{\cL}{\cG}=\aleph_2.
 \]
A general ``crossover result'', proved in~\cite{G1}, is the following \emph{Dichotomy Theorem}:

\begin{all}{Dichotomy Theorem for varieties of algebras}
Let~$\cV$ and~$\cW$ be varieties of algebras, with~$\cV$ locally finite and~$\cW$ finitely generated congruence-distributive \pup{\emph{for example, any variety of lattices}}. If $\Conc\cV\not\subseteq\Conc\cW$, then $\crit{\cV}{\cW}<\aleph_\omega$.
\end{all}

In \cite{GiWe1}, we use techniques of \emph{category theory} to extend this Dichotomy Theorem to much wider contexts, in particular to \emph{congruence-modular} (instead of congruence-distributive) varieties and even to relative congruence classes of \emph{quasivarieties}. Examples of congruence-modular varieties are varieties of \emph{groups} (or even \emph{loops}) or \emph{modules}.

Now comes another mystery. We do not know any critical point, between varieties of algebras with (at most) countable similarity types, equal to~$\aleph_3$, or~$\aleph_4$, and so on: all known critical points are either below~$\aleph_2$ or equal to~$\infty$. How general is that phenomenon?

One half of our main result, Theorem~\ref{T:majorcritpoint}, is the following:

\begin{all}{Dichotomy Theorem for varieties of lattices}
Let~$\cV$ and~$\cW$ be lattice varieties such that every simple member of~$\cW$ contains a prime interval. If $\Conc\cV\not\subseteq\Conc\cW$, then $\crit{\cV}{\cW}\leq\aleph_2$.
\end{all}

In particular, this result gives a solution to Question~$\Qone$ for varieties of lattices where every simple member contains a prime interval. It turns out that the other half of Theorem~\ref{T:majorcritpoint} also solves Question~$\Qtwo$ for those varieties, by proving that $\Conc\cV\subseteq\Conc\cW$ occurs only in the trivial cases, namely~$\cV$ is contained in either~$\cW$ or its dual.

\subsection{Contents of the paper}
Our main idea is the following. It is well known that the assignment $A\mapsto\Conc A$ can be extended, in a standard way, to a \emph{functor}, from all algebras of a given similarity type with their homomorphisms, to the category of all \jzs s and \jzh s: for a homomorphism $f\colon A\to B$ of algebras, $\Conc f$ sends any compact congruence~$\alpha$ of~$A$ to the congruence generated by all pairs $(f(x),f(y))$ where $(x,y)\in\alpha$. Given a finite bounded lattice~$L$, we construct a diagram of \jzs s liftable in a variety of lattices (or bounded lattices)~$\cV$ if and only if either~$L$ or its dual belongs to $\cV$. Then, using a \emph{condensate} \cite[Section 3-1]{GiWe1}, we construct a \jzs~liftable in~$\cV$ if and only if either $L\in\cV$ or $\dual{L}\in\cV$.

The main idea is to start with the \emph{chain diagram~$\vec A$ of~$L$} (cf. Definition~\ref{D:chaindia}). A precursor of that diagram can be found in \cite[Section 4]{G2}. The chain diagram can be described in the following way. The arrows are the inclusion maps. At the bottom of the diagram we put $\set{0,1}$, the sublattice of~$L$ consisting of the bounds of~$L$. On the next level we put all three-element and four-element chains of~$L$ with extremities~$0$ and~$1$, over two chains we put the sublattice of~$L$ generated by those two chains, finally at the top we put the lattice~$L$ itself.

The diagram $\Conc\circ\vec A$ is liftable in any variety that contains either~$L$ or its dual; we do not know any counterexample of the converse yet. However, given a lifting~$\vec B$ of $\Conc\circ\vec A$, with all morphisms of~$\vec B$ being inclusion maps, if all lattices of~$\vec B$ that correspond to chains in~$\vec A$ contain ``congruence chains'' (cf. Definition~\ref{D:congruencechain}) with the same extremities, we also require those chains be ``direct'' (cf. Definition~\ref{D:congruencechain}). For example, if~$A_i=\set{0,x,1}$ is a chain, then~$B_i=\set{u,y,v}$ is also a chain (with $u<y<v$), the congruence $\Theta_{A_i}(0,x)$ corresponds to the congruence $\Theta_{B_i}(u,y)$, and the congruence $\Theta_{A_i}(x,1)$ corresponds to the congruence $\Theta_{B_i}(y,v)$. Under these assumptions we construct a sublattice of the top member of~$\vec B$ isomorphic to~$L$.

The second step of our construction is to expand the chain diagram in order to force congruence chains to be direct (cf. Lemma~\ref{L:directing}). By gluing the chain diagram of~$L$ and copies of the diagram constructed in Lemma~\ref{L:directing}, we obtain a diagram $\vec A'$ such that whenever~$\vec B$ is a lifting of $\Conc\circ\vec A'$ with enough congruence chains, then~$L$ embeds into a quotient of some lattice of~$\vec B$ or its dual.

The third step (cf. Lemma~\ref{L:complextosimplewithchain}) is to ensure the existence of enough congruence chains. We can construct a diagram $\vec A''$ such that if $\Conc\circ\vec A''$ is liftable in some variety~$\cW$, then either $L\in\cW$ or $L\in\dual{\cW}$. For this step, we need a variety~$\cW$ where every simple lattice has a prime interval (i.e., elements~$u$ and~$v$ such that $u\prec v$).

The last step is to use a \emph{condensate} construction \cite[Section 3-1]{GiWe1} on $\vec A''$ to obtain a lattice $B$ of cardinality $\aleph_2$ such that $\Conc B$ is liftable in~$\cW$ if and only if $\Conc\circ\vec A''$ has a ``partial lifting'' in~$\cW$ (cf. \cite[Theorem~9.3]{G3}). Hence $\Conc B$ is liftable in~$\cW$ if and only if either $L\in\cW$ or $L\in\dual{\cW}$.

We use \cite[Theorem~9.3]{G3} which is a generalization of \cite[Theorem~6.9]{G1}. However, the construction used for the latter would give an upper bound $\aleph_3$ for critical points (or $\aleph_2$ in case the variety~$\cW$ is finitely generated). A classical example illustrating the problem is that the three-element chain is the congruence lattice of a modular lattice, but not the congruence lattice of any finite modular lattice.
In order to work around this problem, we used an object introduced in \cite{G3}, called \emph{gamp}. The category of gamps of lattices behave similarly to a finitely generated variety of lattices.

Gamps are certain partial structures endowed with semilattice-valued distances (cf. \cite[Sections~4, 5, and 6]{G3}). A \emph{lattice partial lifting} (cf. \cite[Definition~6.14]{G3}) is a diagram of gamps of lattices that satisfies just enough properties to make the constructions in Sections~\ref{S:ChainDiag} and~\ref{S:ALargerDiag} possible. The chain diagram of a lattice and its uses can be illustrated by the following correspondences:
\begin{align*}
\text{$\bB$ is a gamp of~$\cW$} &\quad\rightleftharpoons\quad \text{$B=B^*$ is a lattice in~$\cW$},\\
\widetilde B &\quad\rightleftharpoons\quad \Conc B,\\
\delta_{\bB}(x,y) &\quad\rightleftharpoons\quad \Theta_B(x,y),\\
\text{$\bgg$ is a morphism of gamps} &\quad\rightleftharpoons\quad \text{$g$ is a homomorphism of lattices},\\
\widetilde g &\quad\rightleftharpoons\quad \Conc g,\\
\text{$\vec\bB$ is partial lifting of~$\vec A$ in~$\cW$} &\quad\rightleftharpoons\quad \text{$\vec B$ is a lifting of~$\vec A$ in~$\cW$.}
\end{align*}

The chain diagram also makes it possible to prove (cf. Theorem~\ref{T:nofunctor}) that if there is a functor $\Psi\colon\cV\to\cW$ such that $\Conc\circ\Psi$ is equivalent to $\Conc$, then either $\cV\subseteq\cW$ or $\cV\subseteq\dual{\cW}$. The functor~$\Psi$ itself does not need to be equivalent to either inclusion or dualization, however we prove in Theorem~\ref{T:nofunctor} that this holds ``up to congruence-preserving extensions''.

\section{Basic concepts}\label{S:Basic}

We denote by $\two$ the two-element lattice, or poset (i.e., partially ordered set), or \jzs, depending of the context. We denote by $\lh C=\card C - 1$ the \emph{length} of a finite chain $C$.

We denote the \emph{range} of a function $f\colon X\to Y$ by $\rng f=\setm{f(x)}{x\in X}$. We use basic set-theoretical notation, for example~$\omega$ is the first infinite ordinal, and also the set of all nonnegative integers; furthermore, $n=\set{0,1,\dots,n-1}$ for every nonnegative integer~$n$. By ``countable'' we will always mean ``at most countable''.

Let $X$, $I$ be sets, we often denote $\vec x=(x_i)_{i\in I}$ an element of $X^I$. In particular, for~$n<\omega$, we denote by~$\vec x=(x_0,\dots,x_{n-1})$ an~$n$-tuple of $X$. Let $m\le n\le\omega$. Let $\vec x$ be an $n$-tuple of~$X$, we denote by $\vec x\res m$ the $m$-tuple $(x_k)_{k<m}$.

Given an algebra~$A$, we denote by $\zero_A$ (resp., $\one_A$) the smallest (resp. largest) congruence of $A$. Given $x,y$ in~$A$, we denote by $\Theta_A(x,y)$ the smallest congruence that identifies $x$ and $y$, and we call it a \emph{principal congruence}. A congruence of~$A$ is \emph{finitely generated} if it is a finite join of principal congruences. The finitely generated congruences of~$A$ are exactly the \emph{compact} elements of the algebraic lattice~$\Con A$. The set~$\Conc A$ of all compact congruences of~$A$ is a \jz-subsemilattice of~$\Con A$.

The \emph{kernel} $\ker f=\setm{(x,y)\in A^2}{f(x)=f(y)}$ is a congruence of~$A$, for any morphism of algebras $f\colon A\to B$. A \emph{congruence-preserving extension} of an algebra~$A$ is an algebra~$B$ that contains~$A$ such that any congruence of~$A$ has a unique extension to~$B$. Equivalently, $\Conc f$ is an isomorphism, where $f\colon A\to B$ denotes the inclusion map.

We denote by $\At B$ the set of \emph{atoms} (that is, the elements that cover~$0$) of a Boolean lattice~$B$.

We denote by $\dual{L}$ the \emph{dual} of a lattice~$L$, that is, the lattice with the same universe as~$L$ and reverse ordering. Obviously, $\Conc\dual L=\Conc L$. If~$\cV$ is a variety of lattices, we denote by~$\dual{\cV}$ the class of all duals of lattices in~$\cV$. It is also the variety that satisfies all dual identities of those satisfied in~$\cV$. The assignment $L\mapsto\dual{L}$ is functorial. We denote by $\dual{\vec L}$ the dual of a diagram~$\vec L$ of lattices (that is, we change each lattice of~$\vec L$ to its dual while keeping the same transition maps).

Let $X$ be set, let $\kappa$ be a cardinal. We denote by~$\Pow(X)$ the powerset of~$X$ and we set:
\[
[X]^{<\kappa}=\setm{Y\in\Pow(X)}{\card Y<\kappa}.
\]
The set $X$ is \emph{$\kappa$-small} if $\card X<\kappa$.

We identify a class~$\cK$ of algebras with the category of algebras of $\cK$ together with homomorphisms of algebras.

Let $\cK$ be a class of algebras on the same similarity type. The \emph{variety generated by $\cK$}, denoted by $\Var\cK$, is the smallest variety containing $\cK$. If $K$ is an algebra, we simply denote $\Var K$ instead of $\Var\set{K}$. A variety~$\cV$ is \emph{finitely generated} if there exists a finite class~$\cK$ of finite algebras such that $\cV=\Var\cK$. Equivalently, $\cV=\Var K$ for a finite algebra~$K$.

Given a poset $P$ and subsets~$X$ and $Q$ of $P$ we set:
\[
 Q\dnw X=\setm{q\in Q}{(\exists x\in X)(q\le x)}.
\]
In case $X=\set{a}$ is a singleton, then we shall write $Q\dnw a$ instead of~$Q\dnw\set{a}$. In case $P=Q$ we shall write $\dnw X$ instead of $P\dnw X$.

Given a poset $P$, we denote by $\Max P$ the set of all maximal elements of $P$, and we set $P^==P-\Max P$. Given $p,q\in P$, we say that~$q$ \emph{covers}~$p$, in notation $p\prec q$, if $p<q$ and there is no element $r\in P$ with $p<r<q$.

A poset $P$ is \emph{directed} if for all $x,y\in P$ there exists $z\in P$ such that $z\ge x,y$. A subset $Q$ of a poset $P$ is a \emph{lower subset} if $Q=P\dnw Q$, \emph{spanning} if $P=P\dnw Q=P\upw Q$. In case~$P$ has a least element~$0$ and a largest element~$1$, $Q$ is spanning if and only if $\set{0,1}\subseteq Q$. For posets~$P$ and~$Q$, a map $f\colon P\to Q$ is \emph{isotone} if $x\leq y$ implies that $f(x)\leq f(y)$ for all $x,y\in P$.

\section{The chain diagram of a lattice}\label{S:ChainDiag}

This section and the next one require familiarity with \emph{gamps} (cf. \cite[Sections~4, 5, and 6]{G3}). All \emph{gamps} are \emph{gamps of lattices} (cf. \cite[Definition~6.2]{G3}), all \emph{partial lifting} are \emph{lattice partial lifting} (cf. \cite[Definition~6.14]{G3}).

In a few words, a \emph{gamp} is a 4-tuple $\bA=(A^*,A,\delta_{\bA},\widetilde A)$, where $A^*\subseteq A$ are partial lattices, $\widetilde A$ is a \jzs, and $\delta_A\colon A^2\to\widetilde A$ is a distance compatible with $\vee$ and $\wedge$. We also require some additional properties (cf. \cite[Definitions~6.1 and~6.2]{G3}). A \emph{morphism of gamps from $\bA$ to $\bB$} is an ordered pair $\bff=(f,\widetilde f)$, where $f\colon A\to B$ is a morphism of partial latices, $\widetilde f\colon\widetilde A\to\widetilde B$ is a \jzs, and $\delta_{\bB}(f(x),f(y))=\widetilde f(\delta_{\bA}(x,y))$ for all $x,y\in A$. 

If $A$ is a lattice, then $\GA(A)=(A,A,\Theta_A,\Conc A)$ is a gamp. If $f\colon A\to B$ is a morphism of lattices, then $\GA(f)=(f,\Conc f)$ is a morphism of gamps. This defines a functor from the category of lattices to the category of gamps (cf. Definition 6.8). We also define a functor~$\CG$ from the category of gamps to the category of \jzs s by $\CG(\bA)=\widetilde A$ for each gamp $\bA=(A^*,A,\delta_{\bA},\widetilde A)$, and $\CG(\bff)=\widetilde f$ for each morphism $\bff=(f,\widetilde f)$ of gamps. Note that $\CG\circ\GA=\Conc$. A \emph{partial lifting} is a lifting, with respect to the $\CG$ functor, which satisfies a few additional properties \cite[Definition~6.14]{G3}.

Every lattice~$A$ can be identified with the associated gamp~$\GA(A)$; hence gamps generalize lattices. Once this identification is done, the~$\CG$ functor ``extends'' the $\Conc$ functor, and partial liftings generalize $\Conc$-liftings.

The aim of this section is, given a partial sublattice~$K$ of a lattice~$L$, to construct a diagram of lattices~$\vec A$, such that whenever $\Conc\circ\vec A$ has a partial lifting in a variety~$\cV$, then~$K$ is a partial sublattice of some $L'\in\cV\cup\dual{\cV}$.

We shall construct in Definition~\ref{D:chaindia} a diagram~$\vec A$, called the chain diagram of~$K$ in~$L$. In Theorem~\ref{T:good-liftingdia-imply-lifting} we shall prove that if $\Conc\circ\vec A$ has a partial lifting with enough direct congruence chains (cf. Definition~\ref{D:congruencechain}) then~$K$ is a partial sublattice of a partial lattice in the partial lifting. In Lemma~\ref{L:directing} we shall give diagrams that ``force'' congruence chains to be direct, then in Lemma~\ref{L:exists-good-diagram}, using Definition~\ref{D:productdiagram} and Lemma~\ref{L:diagramextension}, we shall glue all these diagrams together to obtain a diagram $\vec A'$, such that whenever $\vec \bB$ is a partial lifting of $\Conc\circ\vec A'$ with enough congruence chains, then~$K$ is a partial sublattice of a quotient of either $\bB_1$ or its dual.

\begin{remark}
The \emph{dual} of a partial lattice~$B$ is the partial lattice $\dual{B}$, with the same set of elements~$B$, $\Def_{\vee}(\dual{B})=\Def_{\wedge}(B)$, $\Def_{\wedge}(\dual{B})=\Def_{\vee}(B)$, $x\vee_{\dual{B}} y=x\wedge_B y$ for all $(x,y)\in\Def_{\vee}(\dual{B})$, and $x\wedge_{\dual{B}} y=x\vee_B y$ for all $(x,y)\in\Def_{\wedge}(\dual{B})$.

The \emph{dual} of a gamp~$\bB$ of lattices is $\dual{\bB}=(\dual{B^*},\dual{B},\delta_{\bB},\widetilde B)$.
\end{remark}

Let~$\cC$ be a set of nonempty finite chains. We put:
\begin{align*}
\JC(\cC)&=\set{\emptyset}\cup\setm{\set{C}}{C\in\cC},\\
\KC(\cC)&=\JC(\cC)\cup\setm{\set{C,D}\subseteq \cC}{\text{either }\lh(C)=\lh(D)=2\text{ or } C\subseteq D},\\
\IC(\cC)&=\KC(\cC)\cup\set{\top},
\end{align*}
ordered by inclusion on $\KC(\cC)$ and with $P\le \top$ for each $P\in \IC(\cC)$. Put
\begin{equation*}
E_P=
\begin{cases}
C &\text{if $P=\set{C}$}\\
\set{0,1} &\text{if $P=\emptyset$}
\end{cases},\quad\text{for each $P\in \JC(\cC)$}.
\end{equation*}
Let $e_{\emptyset,\set{C}}\colon \set{0,1}\to C$ be the bounds-preserving morphism, and $e_{P,P}=\id_{E_P}$ for each $P\in\JC(\cC)$. Thus $\vec E(\cC)=\famm{E_P,e_{P,Q}}{P\le Q\text{ in }\JC(\cC)}$ is a diagram of lattices, with $0,1$-homomorphisms of lattices.

The following definition describes the way our diagrams will be glued together.

\begin{definition}\label{D:productdiagram}
Let $T$ be a finite nonempty set, let~$I$ be a finite poset, let~$J$ be a lower subset of~$I$. Let~$\vec{A}^t=\famm{A^t_i,f^t_{i,j}}{i\le j\text{ in }I}$ be a diagram of algebras for all $t\in T$, such that~$\vec A^s\res J=\vec A^t\res J$ for all $s,t\in T$, that is
\begin{enumerate}
\item The equality~$A^s_j=A^{t}_j$ holds for all $s,t\in T$ and all $j\in J$.
\item The equality $f^s_{i,j}=f^{t}_{i,j}$ holds for all $s,t\in T$ and all $i\le j$ in~$J$.
\end{enumerate}
Set
\begin{equation*}
A_i=
\begin{cases}
A_i^t &\text{for any $t\in T$, if $i\in J$}\\
\prod_{t\in T}A_i^t &\text{if $i\in I-J$}
\end{cases},\quad\text{for each $i\in I$.}
\end{equation*}
If $i\in J$, denote $\pi_i^t=\id_{A_i}=\id_{A_i^t}$. If $i\in I-J$, denote $\pi_i^t\colon A_i\to A_i^t$ the canonical projection. Let $f_{i,j}\colon A_i\to A_j$ be defined by
\begin{align*}
f_{i,j}\colon A_i &\to A_j,\\
x &\mapsto
\begin{cases}
f_{i,j}^t(x) &\text{for any $t\in T$, if $j\in J$}\\
(f_{i,j}^t(\pi_i^t(x)))_{t\in T} &\text{if $j\not\in J$}.
\end{cases},\quad\text{for all $i\le j$ in~$I$.}
\end{align*}
Then~$\vec A=\famm{A_i,f_{i,j}}{i\le j\text{ in~$I$}}$ is the \emph{product of $(\vec{A}^t)_{t\in T}$ over~$J$}, and $\vec\pi^t=(\pi_i^t)_{i\in I}\colon \vec A\to  \vec A^t$ is \emph{the canonical projection}.
\end{definition}

\begin{remark}\label{R:productdiagram}
Let~$\cK$ be a class of bounded lattices closed under finite products. Consider the objects of Definition~\ref{D:productdiagram}, assume that $\vec{A}^t$ is a diagram in~$\cK$ for each $t\in T$. Then~$\vec A$ the product of $(\vec{A}^t)_{t\in T}$ over~$J$ is a diagram in~$\cK$.

Let $i\in I$ such that $A_i^t$ is finite for each $t\in T$, then $A_i$ is finite.
\end{remark}

Using the construction of the product in Definition~\ref{D:productdiagram} we see that $\vec A\res J=\vec A^t\res J$ for each $t\in T$.

The following lemma gives a way to extend diagrams indexed by $\IC(\cC)$ to $\IC(\cC')$ for $\cC\subseteq\cC'$.

\begin{lemma}\label{L:diagramextension}
Let~$\cK$ be a class of bounded lattices, let $\cC\subseteq \cC'$ be sets of finite chains in~$\cK$. Then $\IC(\cC)\subseteq \IC(\cC')$ and $\JC(\cC)\subseteq \JC(\cC')$. Furthermore, consider a diagram $\vec B=\famm{B_P,g_{P,Q}}{P\le Q\text{ in }\IC(\cC)}$ in~$\cK$ \pup{with $0,1$-homomorphisms of lattices}. If $\vec B\res \JC(\cC)=\vec E(\cC)$, then there exists a diagram $\vec B'$ of~$\cK$, indexed by $\IC(\cC')$, such that $\vec B'\res \IC(\cC)=\vec B$ and $\vec B'\res \JC(\cC')=\vec E(\cC')$.
\end{lemma}

\begin{proof}
It is sufficient to establish the result in case $\cC'=\cC\cup\set{C'}$, for some chain~$C'$. For $P\in \IC(\cC')$, let:
\begin{align*}
B'_P=
\begin{cases}
B_P &\text{if $P\in \IC(\cC)$}\\
B_{\set{C}}=C  &\text{if $P=\set{C,C'}$, with~$C\in\cC$,}\\
C'  &\text{if $P=\set{C'}$}
\end{cases}
\end{align*}
For a finite chain~$C$, put 
\begin{align*}
f_C\colon C &\to \two,\\
x &\mapsto
\begin{cases}
0 & \text{if $x<1_C$,}\\
1 & \text{if $x=1_C$.}
\end{cases}
\end{align*}
Let $P\le Q$ in $\IC(\cC')$. Set $g'_{P,P}=\id_{B'_P}$, and if $P<Q$,
\begin{align*}
g'_{P,Q}=
\begin{cases}
g_{P,Q} &\text{if $P,Q\in \IC(\cC)$,}\\
g_{\set{C},\top}  &\text{if $P=\set{C,C'}$ and $Q=\top$, with~$C\in\cC$},\\
e_{\emptyset,\set{C'}} &\text{if $P=\emptyset$ and $Q=\set{C'}$},\\
e_{\emptyset,\set{C}} &\text{if $P=\emptyset$ and $Q=\set{C,C'}$, with~$C\in\cC$},\\
e_{\emptyset,\set{C}}\circ f_{C'} &\text{if $P=\set{C'}$ and $Q=\set{C,C'}$, with~$C\in\cC$},\\
\id_{C} &\text{if $P=\set{C}$ and $Q=\set{C,C'}$, with~$C\in\cC$},\\
g_{\emptyset,\top}\circ f_{C'} &\text{if $P=\set{C'}$ and $Q=\top$.}
\end{cases}
\end{align*}

The proof that $\vec B'=\famm{B'_P,g'_{P,Q}}{P\le Q\text{ in }\IC(\cC')}$ is a diagram in~$\cK$ is straightforward. Moreover, by construction $\vec B'\res \IC(\cC)=\vec B$ and $\vec B'\res \JC(\cC')=\vec E(\cC')$.
\end{proof}

We refer to \cite[Definition~6.2]{G3} for the definition of a \emph{chain} and the definition of \emph{cover} in a gamp of lattices. The following is proved in \cite[Lemma~6.4]{G3}
\begin{lemma}\label{L:presqueordre}
Let $x_0<\dots<x_n$ be a chain of a strong gamp of lattices~$\bB$. The equalities $x_i\wedge x_j=x_j\wedge x_i=x_i$ and $x_i\vee x_j=x_j\vee x_i=x_j$ hold in~$B$ for all $i\le j\le n$. Moreover the following statements hold:
\begin{align}
\delta_{\bB}(x_k,x_{k'})&\le \delta_{\bB}(x_i,x_j),\quad\text{for all $i\le k\le k'\le j\le n$}.
\label{E:po1}\\
\delta_{\bB}(x_i,x_j)&=\bigvee_{i\le k<j}\delta_{\bB}(x_k,x_{k+1}),\quad\text{for all $i\le j\le n$}.\label{E:po2}
\end{align}
\end{lemma}

\begin{definition}\label{D:congruencechain}
Let~$\bB$ be a gamp of lattices such that $\widetilde B$ is a finite Boolean lattice. A chain $x_0<x_1<\dots<x_n$ of~$\bB$ is a \emph{congruence chain} of~$\bB$ if there exists a bijection $\sigma\colon n\to\At\widetilde B$ such that $\delta_{\bB}(x_k,x_{k+1})=\sigma(k)$ for each $k<n$.

Let~$C=\set{c_0,\dots,c_n}$ be a chain with $c_0<c_1<\dots<c_n$ and let $\xi\colon\widetilde B\to\Conc C$ be an isomorphism. A congruence chain $x_0<\dots<x_n$ of~$\bB$ is \emph{direct for $(\xi,C)$} if $\xi(\delta_{\bB}(x_k,x_{k+1}))=\Theta_C(c_k,c_{k+1})$ for all $k<n$. We simply say that $x_0<x_1<\dots<x_n$ is a \emph{direct congruence chain of~$\bB$} in case $\xi$ and~$C$ are both understood.

In both cases $x_0$ and $x_n$ are the \emph{extremities} of the congruence chain.

Given a lattice~$B$ with a finite Boolean congruence lattice, a \emph{congruence chain of~$B$} is a congruence chain of $\GA(B)$.
\end{definition}

Given a lattice~$B$ with a finite Boolean congruence lattice, a chain $x_0<x_1<\dots<x_n$ is a congruence chain of~$B$ if and only if there exists a bijection $\sigma\colon n\to\At\widetilde B$ such that $\Theta_B(x_k,x_{k+1})=\sigma(k)$ for each $k<n$.

\begin{remark}\label{R:congruencechain}
Let~$C$ be a chain of length~$2$, let $\bB$ be a strong gamp of lattices, and let $\xi\colon\widetilde B\to\Conc C$ be an isomorphism. A congruence chain $x_0<x_1<x_2$ of~$\bB$ is either direct or dually direct, that is, either  $x_0<x_1<x_2$ is a direct congruence chain of~$\bB$ or $x_2<_{\dual{\bB}}x_1<_{\dual{\bB}}x_0$ is a direct congruence chain of~$\dual{\bB}$.

A finite chain~$C$ is a congruence chain of a lattice~$B$ if an only if~$B$ is a congruence-preserving extension of~$C$. In particular if~$B$ is a finite distributive lattice, then every maximal chain of~$B$ is a congruence chain of~$B$.
\end{remark}

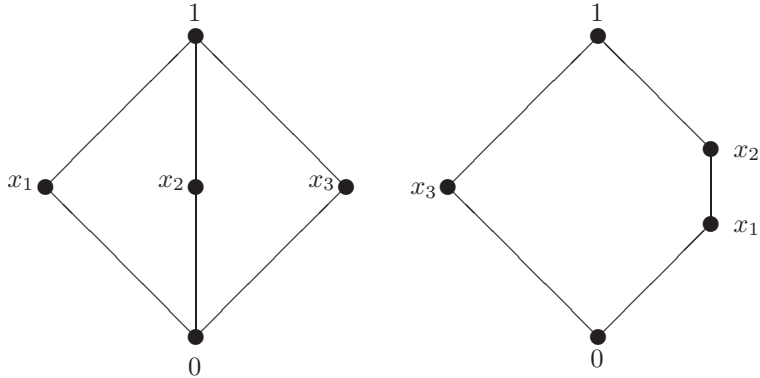
\begin{figure}[here,bottom,top]
\caption{The lattices $M_3$ and $N_5$.}\label{F:treillis_M3_N5}
\setlength{\unitlength}{1mm}
\begin{picture}(50,55)(-5,-5)
\put(20,0){\line(1,1){20}}
\put(20,0){\line(-1,1){20}}
\put(20,0){\line(0,1){20}}
\put(19,-5){0}
\put(-5,20){$x_1$}
\put(15,20){$x_2$}
\put(35,20){$x_3$}
\put(19,42){1}

\put(0,20){\line(1,1){20}}
\put(20,20){\line(0,1){20}}
\put(40,20){\line(-1,1){20}}

\put(20,0){\circle*{2}}
\put(0,20){\circle*{2}}
\put(20,20){\circle*{2}}
\put(40,20){\circle*{2}}
\put(20,40){\circle*{2}}
\end{picture}\quad
\begin{picture}(50,50)(-5,-5)
\put(20,0){\line(-1,1){20}}
\put(20,0){\line(1,1){15}}
\put(35,15){\line(0,1){10}}
\put(0,20){\line(1,1){20}}
\put(35,25){\line(-1,1){15}}
\put(20,0){\circle*{2}}
\put(0,20){\circle*{2}}
\put(35,15){\circle*{2}}
\put(35,25){\circle*{2}}
\put(20,40){\circle*{2}}
\put(19,-4){$0$}
\put(19,42){$1$}
\put(-5,19){$x_3$}
\put(38,14){$x_1$}
\put(38,24){$x_2$}
\end{picture}
\end{figure}

In our next lemma we shall construct a diagram of lattices that ``forces'' congruence chains to be direct.

\begin{lemma}\label{L:directing}
Let $K=N_5$ or $K=M_3$, with vertices labeled as in the Figure~\textup{\ref{F:treillis_M3_N5}}, let~$\cK$ be a class of bounded lattices closed under finite products such that every bounded sublattice of $K$ belongs to~$\cK$. Let~$C_1,C_2,C_3\in\cK$ be distinct finite chains with extremities~$0$ and~$1$ such that both~$C_1$ and~$C_2$ have length~$2$ while either~$C_3$ has length~$2$ or~$C_1,C_2\subseteq C_3$. Put $\cC=\set{C_1,C_2,C_3}$. Then there exists a diagram~$\vec A$ of finite lattices of~$\cK$ with $0,1$-lattice homomorphisms, indexed by $\IC(\cC)$, such that:
\begin{enumerate}
\item The equality $\vec A\res \JC(\cC)=\vec E(\cC)$ holds.
\item Let~$\vec\bB$ be a partial lifting of $\Conc\circ\vec A$, let $\vec\xi\colon\CG\circ\vec\bB\to\Conc\circ\vec A$ be a natural equivalence, and let $u\not=v$ in~$B_{\emptyset}^*$. If $\bB_{\set{C_k}}$ contains a direct congruence chain with extremities $g_{\emptyset,\set{C_k}}(u)$ and $g_{\emptyset,\set{C_k}}(v)$ for each $k\in\set{1,2}$, then every congruence chain of $\bB_{\set{C_3}}$ with extremities $g_{\emptyset,\set{C_3}}(u)$ and $g_{\emptyset,\set{C_3}}(v)$ is direct.
\end{enumerate}
Note that the indexing set $\IC(\cC)$ is the usual cube $\set{0,1}^3$.
\end{lemma}

\begin{proof}
We can assume that~$C_1=\set{0,x_1,1}$ and~$C_2=\set{0,x_2,1}$. Let $0=y_0<y_1<y_2<\dots<y_n=1$ be the elements of~$C_3$. Put $D_3=\set{0,x_3,1}$ and set
 \[
 T=\setm{t}{t\colon C_3\tosurj D_3\text{ is isotone and surjective}}.
 \]
Let $t\in T$. Put:
\begin{align*}
A_\emptyset^t&=A_\emptyset=\set{0,1},\\
A_{\set{C_i}}^t&=A_{\set{C_i}}=C_i &  & \text{for $1\le i\le 3$},\\
A_{\set{C_i,C_j}}^t&=\set{0,x_i,x_j,1} & & \text{for $1\le i<j\le 3$},\\
A_{\top}^t&=K.
\end{align*}
Note that $A_P^t$ is a sublattice of $K$, in particular it belongs to $\cK$, for all $P\in\cC$.

For $P\le Q$ in $\IC(\cC)$, let $f_{P,Q}^t\colon A_P^t\to A_Q^t$ be the inclusion map if $P\not=\set{C_3}$ and $Q\not=\set{C_3}$, otherwise let:
\begin{align*}
f_{P,Q}^t=
\begin{cases}
\id_{A_P} &\text{if $P=Q=\set{C_3}$,}\\
t &\text{if $P=\set{C_3}$ and $Q>P$,}\\
e_{\emptyset,C_3} &\text{if $P=\emptyset$ and $Q=\set{C_3}$.}
\end{cases}
\end{align*}
Let $P<Q<R$ in $\IC(\cC)$. As all the maps involved are $0,1$-homomorphisms and~$A^t_{\emptyset}=\set{0,1}$, if $P=\emptyset$ then $f_{P,R}^t=f_{Q,R}^t\circ f_{P,Q}^t$. Now assume that $P>\emptyset$. If $P\not=\set{C_3}$, then $f_{P,Q}^t$, $f_{Q,R}^t$, and $f_{P,R}^t$ are all inclusion maps, thus $f_{P,R}^t=f_{Q,R}^t\circ f_{P,Q}^t$. If $P=\set{C_3}$, then $f_{P,Q}^t=t$, the morphism $f_{Q,R}^t$ is the inclusion map, and $f_{P,R}^t=t$, so $f_{P,R}^t=f_{Q,R}^t\circ f_{P,Q}^t$. Thus $\vec A^t=\famm{A_P^t,f_{P,Q}^t}{P\le Q\text{ in }\IC(\cC)}$ is a diagram of finite lattices of~$\cK$. Moreover by construction $\vec A^t\res \JC(\cC)=\vec E(\cC)$.

Let $\vec A=\famm{A_P,f_{P,Q}}{P\le Q\text{ in }\IC(\cC)}$ be the product of $(\vec A^t)_{t\in T}$ over $\JC(\cC)$ (cf. Definition~\ref{D:productdiagram}). Hence $\vec A\res \JC(\cC)=\vec E(\cC)$ and it follows from Remark~\ref{R:productdiagram} that~$\vec A$ is a diagram of finite lattices in~$\cK$.

Let $\vec\bB=\famm{\bB_P,\bgg_{P,Q}}{P\le Q\text{ in }\IC(\cC)}$ be a partial lifting of $\Conc\circ\vec A$. We can assume that $\vec\xi$ is the identity, that is, $\CG\circ\vec\bB=\Conc\circ\vec A$. Set $\delta_P=\delta_{\bB_P}$, for each $P\in\IC(\cC)$. Let $u\not=v$ in~$B_{\emptyset}^*$ with $u\wedge v=u$. Assume that $\bB_{\set{C_k}}$ contains a direct congruence chain $g_{\emptyset,\set{C_k}}(u)<x'_k<g_{\emptyset,\set{C_k}}(v)$ for $k=1,2$. Let $g_{\emptyset,\set{C_3}}(u)=y'_0<y'_1<\dots<y'_n=g_{\emptyset,\set{C_3}}(v)$ be a congruence chain of $\bB_{\set{C_3}}$.

Let $\sigma\colon\set{0,1,\dots,n-1}\to\set{0,1,\dots,n-1}$ be a bijection such that:
\begin{equation}\label{E:Eqn1}
\Theta_{\set{C_3}}(y_i,y_{i+1})= \delta_{\set{C_3}}(y'_{\sigma(i)},y'_{\sigma(i)+1}),\quad\text{for all $i<n$.}
\end{equation}
Assume that $\sigma$ is not the identity map. Let $i$ be minimal such that $\sigma(i)\not=i$. The minimality of $i$ implies that $\sigma(k)=k$ for all $k<i$, thus:
\[
\Theta_{C_3}(y_0,y_{i}) =\bigvee_{k<i}\Theta_{C_3}(y_k,y_{k+1})=\bigvee_{k<i} \delta_{\set{C_3}}(y'_k,y'_{k+1}).
\]
It follows from Lemma~\ref{L:presqueordre} that the following equality holds:
\begin{equation}
 \Theta_{C_3}(y_0,y_{i})=\delta_{\set{C_3}} (y'_0,y'_{i}).\label{E:Eqn2}
\end{equation}

Set $j=\sigma^{-1}(i)$. Then
\begin{equation}
\Theta_{C_3}(y_j,y_{j+1})= \delta_{\set{C_3}} (y'_i,y'_{i+1}).\label{E:Eqn3}
\end{equation}
It follows from Lemma~\ref{L:presqueordre},~\eqref{E:Eqn2}, and~\eqref{E:Eqn3} that the following equality holds:
\begin{equation}
\delta_{\set{C_3}} (y'_0,y'_{i+1})
=\Theta_{C_3}(y_0,y_{i})\vee\Theta_{C_3}(y_j,y_{j+1}).\label{E:Eqn2et3}
\end{equation}

Using the minimality of $i$, we obtain that $i<j$. Hence $\sigma$ defines by restriction a bijection from $\set{i,i+1,\dots,n-1}-\set{j}$ onto $\set{i+1,i+2,\dots,n-1}$, and so
\begin{align*}
\Theta_{C_3} (y_i,y_{j}) \vee \Theta_{C_3}(y_{j+1},y_{n})
&=\bigvee\Setm{\Theta_{C_3}(y_k,y_{k+1})}{k\in \set{i,i+1,\dots,n-1}-\set{j}}\\
&=\bigvee\Setm{\delta_{\set{C_3}}(y'_{\sigma(k)},y'_{\sigma(k)+1})}{k\in \set{i,i+1,\dots,n-1}-\set{j}}\\
&=\bigvee\Setm{\delta_{\set{C_3}}(y'_{s},y'_{s+1})}{s\in \set{i+1,i+2,\dots,n-1}}.
\end{align*}
It follows from Lemma~\ref{L:presqueordre} that the following equation is satisfied:
\begin{equation}
\delta_{\set{C_3}} (y'_{i+1},y'_{n})=\Theta_{C_3} (y_i,y_{j})\vee \Theta_{C_3}(y_{j+1},y_{n}). \label{E:Eqn4}
\end{equation}
Let:
\begin{align*}
t\colon C_3 &\to D_3\\
y_k &\mapsto
\begin{cases}
0 & \text{if $k\le i$}\\
x_3 & \text{if $i< k\le j$}\\
1 & \text{if $j<k$}
\end{cases}
,\quad\text{for all $k<n$.}
\end{align*}
Let $\vec \pi=(\pi_P)_{P\in \IC(\cC)}\colon\vec A\to\vec A^t$ be the canonical projection. The vector $\vec\chi=\Conc\circ\vec\pi$ is an ideal-induced natural transformation from $\CG\circ\vec\bB$ to $\Conc\circ\vec A^t$. Denote $\delta_P'=\chi_P\circ\delta_P$ for each $P\in\IC(\cC)$. Notice that~$A_P^t=A_P$, $\pi_P=\id_{A_P}$, $\chi_P=\id_{\Conc A_P}$, and $\delta_P'=\delta_P$ for each $P\in\JC(\cC)$. Let $P\le Q$ in $\IC(\cC)$ with $P\in\JC(\cC)$. Let $a,b\in B_P$. The following equations are satisfied:
\begin{align}
\delta_Q'(g_{P,Q}(a),g_{P,Q}(b))
&=\chi_Q( (\Conc f_{P,Q})( \delta_P(a,b)))\notag\\
&=(\Conc f_{P,Q}^t)( \chi_P( \delta_P(a,b)))\notag\\
&=(\Conc f_{P,Q}^t)(\delta_P(a,b)).\label{Eq:simplchidelta}
\end{align}

As $g_{\emptyset,\set{C_k}}(u)<x'_k<g_{\emptyset,\set{C_k}}(v)$ is a direct congruence chain of $\bB_{\set{C_k}}$, we obtain
\begin{equation}\label{Eq:dct}
\delta_{\set{C_k}}(g_{\emptyset,\set{C_k}}(u),x'_k)=\Theta_{C_k} (0,x_k).
\end{equation}

Let $k\in\set{1,2}$, let $Q\ge\set{C_k}$ in $\IC(\cC)$. The following equalities hold:
\begin{align*}
\delta_Q'(g_{\emptyset,Q}(u),g_{\set{C_k},Q}(x'_k))&=(\Conc f_{\set{C_k},Q}^t)(\delta_{\set{C_k}}(g_{\emptyset,\set{C_k}}(u),x'_k)) & &\text{by~\eqref{Eq:simplchidelta}.}\\
&=(\Conc f_{\set{C_k},Q}^t)(\Theta_{C_k} (0,x_k))& &\text{by~\eqref{Eq:dct}.}\\
&=\Theta_{A_Q^t} (0,x_k),
\end{align*}
hence
\begin{equation}\label{eq:prelem}
\delta_Q'(g_{\emptyset,Q}(u),g_{\set{C_k},Q}(x'_k))=\Theta_{A_Q^t}(0,x_k),\quad\text{for $k=1,2$, and $Q\ge\set{C_k}$.}
\end{equation}

Let $k\in\set{1,2}$ and set $Q=\set{C_k,C_3}$. Put $u'=g_{\set{C_3},Q}(y'_0)=g_{\emptyset,Q}(u)$, $v'=g_{\set{C_3},Q}(y'_n)=g_{\emptyset,Q}(v)$, and $a=g_{\set{C_3},Q}(y'_{i+1})$. The following equalities hold:
\begin{align*}
\delta_Q'(u',a)
&=\delta_Q'(g_{\set{C_3},Q}(y'_0),g_{\set{C_3},Q}(y'_{i+1}))\\
&=(\Conc t)(\delta_{\set{C_3}}(y'_0,y'_{i+1})) &&\text{by~\eqref{Eq:simplchidelta}, as $f_{\set{C_3},Q}^t=t$.}\\
&=(\Conc t)( \Theta_{C_3}(y_0,y_i)\vee\Theta_{C_3}(y_j,y_{j+1})) &&\text{by~\eqref{E:Eqn2et3}.}\\
&=\Theta_{A_Q^t}(t(y_0),t(y_i))\vee \Theta_{A_Q^t}(t(y_j),t(y_{j+1}))\\
&=\Theta_{A_Q^t}(x_3,1),
\end{align*}
hence
\begin{equation}\label{E:t(ua)}
\delta_Q'(u',a)= \Theta_{A_Q^t}(x_3,1).
\end{equation}
Similarly, it follows from~\eqref{E:Eqn4} that
\begin{equation}\label{E:t(av)}
\delta_Q'(a,v')= \Theta_{A_Q^t}(0,x_3).
\end{equation}

Put $b=g_{\set{C_k},Q}(x'_k)$. It follows from~\eqref{eq:prelem} that
\begin{equation}\label{E:t(ub)}
\delta_Q'(u',b)= \Theta_{A_Q^t}(0,x_k)= \Theta_{A_Q^t}(x_3,1).
\end{equation}
With a similar argument we obtain:
\begin{equation}\label{E:t(bv)}
\delta_Q'(b,v')= \Theta_{A_Q^t}(0,x_3).
\end{equation}

The equations~\eqref{E:t(ua)} and~\eqref{E:t(ub)} imply:
\[
\delta_Q'(a,b)\subseteq \delta_Q'(a,u')\vee\delta_Q'(u',b)=\Theta_{A_Q^t}(x_3,1).
\]
Similarly, from~\eqref{E:t(av)} and~\eqref{E:t(bv)}, we obtain:
\[
\delta_Q'(a,b)\subseteq \delta_Q'(a,v')\vee\delta_Q'(v',b)=\Theta_{A_Q^t}(0,x_3).
\]
Therefore, the following containments hold:
\[
\delta_Q'(a,b) \subseteq \Theta_{A_Q^t}(x_3,1)\cap\Theta_{A_Q^t}(0,x_3)=\zero_{A_Q^t},
\]
that is, $\delta_Q'(g_{\set{C_3},Q}(y'_{i+1}),g_{\set{C_k},Q}(x'_k))=\zero_{A_Q^t}$. Hence
\[
\delta_\top'(g_{\set{C_3},\top}(y'_{i+1}),g_{\set{C_k},\top}(x'_k))=\zero_{A_\top^t},\quad\text{for all $k\in\set{1,2}$.}
\]
So $\delta_\top'(g_{\set{C_1},\top}(x'_1),g_{\set{C_2},\top}(x'_2))=\zero_{A_\top^t}$. Put $R=\set{C_1,C_2}$, set $x''_1=g_{\set{C_1},R}(x'_1)$ and $x''_2=g_{\set{C_2},R}(x'_2)$. The following equalities hold:
\begin{align*}
\zero_{A_\top^t}&=\delta_\top'(g_{\set{C_1},\top}(x'_1),g_{\set{C_2},\top}(x'_2))\\
&=(\Conc f_{R,\top}^t)(\delta_R'(g_{\set{C_1},R}(x'_1),g_{\set{C_2},R}(x'_2)))\\
&=(\Conc f_{R,\top}^t)(\delta_R'(x''_1,x''_2)).
\end{align*}
Moreover, as $f_{R,\top}^t$ is an embedding, the map $\Conc f_{R,\top}^t$ separates~$0$, and so
\begin{equation}\label{eq:x1isx2}
\delta_R'(x''_1,x''_2) = \zero_{A_R^t}.
\end{equation}
The following equalities hold:
\begin{align*}
\Theta_{A_R^t}(0,x_1)&=\delta_R'( g_{\emptyset,R}(u),x''_1) & &\text{by~\eqref{eq:prelem}}\\
&=\delta_R'( g_{\emptyset,R}(u),x''_1)\vee \delta_R'(x''_1,x''_2)& &\text{by~\eqref{eq:x1isx2}}\\
&=\delta_R'( g_{\emptyset,R}(u),x''_2)\vee \delta_R'(x''_1,x''_2)\\
&=\delta_R'(g_{\emptyset,R}(u),x''_2)& &\text{by~\eqref{eq:x1isx2}}\\
&=\Theta_{A_R^t}(0,x_2)& &\text{by~\eqref{eq:prelem}}.
\end{align*}

For $K=N_5$, the lattice~$A_R^t$ is the three-element chain $0<x_1<x_2<1$. For $K=M_3$, the lattice~$A_R^t$ is the square $0<x_1,x_2<1$. In both cases, $\Theta_{A_R^t}(0,x_1)\not=\Theta_{A_R^t}(0,x_2)$, a contradiction; thus $\sigma$ is the identity map, so the congruence chain $g_{\emptyset,C_3}(u)=y'_0<y'_1<\dots<y'_n=g_{\emptyset,C_3}(v)$ is direct.
\end{proof}

A partial sublattice $B$ of $A$ is \emph{full} if it satisfies the following condition: For all $x,y\in B$, if $x\wedge_A y$ is defined and $x\wedge_A y\in B$ then $x\wedge_B y$ is defined (note that $x\wedge_B y=x\wedge_A y$); there is a similar condition for $\vee$ (cf. \cite[Definition 4.4]{G3}).

\begin{definition}\label{D:chaindia}
Let~$L$ be a nontrivial bounded lattice, let~$\cC$ be a set of spanning finite chains of~$L$. Set~$A_\emptyset=\set{0,1}$, set~$A_\top=L$, and let~$A_P$ be the sublattice of~$L$ generated by $\bigcup_{C\in P}C$ for each $P\in \IC(\cC)-\set{\emptyset,\top}$. Let $f_{P,Q}\colon A_P\to A_Q$ be the inclusion map, for all $P\le Q$ in $\IC(\cC)$. Then $\vec A=\famm{A_P,f_{P,Q}}{P\le Q\text{ in }\IC(\cC)}$ is the \emph{$\cC$-chain diagram of~$L$}.

Let~$K$ be a spanning full partial sublattice of~$L$. We denote by~$\cC_K$ be the set of all spanning chains of~$L$ of length either~$2$ or~$3$  contained in $K$. \emph{The chain diagram of~$K$ in~$L$} is the~$\cC_K$-chain diagram of~$L$. Any $C\in\cC_K$ is a lattice which is a full partial sublattice of $K$.
\end{definition}

Let~$C\in\cC$, hence~$A_{\set{C}}=C$. Thus $\vec A\res  \JC(\cC)=\vec E(\cC)$.

Let $P\in \IC(\cC)-\set{\top}$. If $P=\emptyset$, then~$A_P=\set{0,1}$ is finite. If $P=\set{C}$, then~$A_P=C$ is finite. If $P=\set{C,D}$, with $\lh(C)=2=\lh(D)$ or~$C\subseteq D$, the lattice~$A_P$ generated by~$C\cup D$ is finite and distributive. Hence~$A_P$ is finite and distributive for each $P<\top$ in $\IC(\cC)$.

\begin{remark}
Let~$\cK$ be a class of bounded lattices, assume that $L$ and all bounded lattices generated by two elements (see Figure~\ref{F:lat-AQ}) belong to~$\cK$. Let~$K$ be a spanning full partial sublattice of~$L$, then the chain diagram of~$K$ in~$L$ is a diagram in $\cK$.
\end{remark}

\begin{remark}
Let $\vec A=\famm{A_i,f_{i,j}}{i\le j\text{ in }I}$ be a diagram indexed by a poset~$I$, let $\vec\bB=\famm{\bB_i,\bgg_{i,j}}{i\le j\text{ in }I}$ be a partial lifting of $\Conc\circ\vec A$. We can assume that $\CG\circ\vec\bB=\Conc\circ\vec A$ (cf. \cite[Remark~6.15]{G3}). Let $i<j$ in~$I$, let $x,y$ in~$A_i$, let $x',y'$ in~$B_i$, assume that $\delta_{\bB_i}(x',y')=\Theta_{A_i}(x,y)$. Then:
\begin{align*}
\delta_{\bB_j}(g_{i,j}(x'),g_{i,j}(y'))&=(\Conc f_{i,j})(\delta_{\bB_i}(x',y'))\\
&=(\Conc f_{i,j})(\Theta_{A_i}(x,y))\\
&=\Theta_{A_j}(f_{i,j}(x),f_{i,j}(y)).
\end{align*}
In particular, if both $f_{i,j}$ and $g_{i,j}$ are inclusion maps, then $\delta_{\bB_j}(x',y')=\Theta_{A_j}(x,y)$.
\end{remark}

The following lemma is a first step to prove Theorem~\ref{T:good-liftingdia-imply-lifting}; it handles a few particular cases of meets and joins. It also shows that different congruence chains of a partial lifting are identical provided they correspond to the same element of the chain diagram.

\begin{lemma}\label{L:mainchaindia}
Let~$L$ be a bounded lattice, let~$\cC$ be a set of spanning chains of~$L$, let~$\vec A$ be the~$\cC$-chain diagram of~$L$, and let $\vec\bB=\famm{\bB_P,\bgg_{P,Q}}{P\le Q\text{ in }\IC(\cC)}$ be a partial lifting such that $g_{P,Q}$ is the inclusion map for all $P\le Q$ in $\IC(\cC)$. Let $\vec\xi\colon\CG\circ\vec\bB\to\Conc\circ\vec A$ be a natural equivalence, let $u,v$ in~$B_{\emptyset}$, let $x_1\not=x_2$ in $L-\set{0,1}$. Set~$C_1=\set{0,x_1,1}$ and~$C_2=\set{0,x_2,1}$. Assume $\cC\supseteq\set{C_1,C_2}$, let $u<y_k<v$ be a direct congruence chain of $\bB_{\set{C_k}}$ for each $k\in\set{1,2}$. Then the following statements hold:
\begin{enumerate}
\item $\xi_{\set{C_1,C_2}}(\delta_{\bB_{\set{C_1,C_2}}}(y_1,y_2))=\Theta_{A_{\set{C_1,C_2}}}(x_1,x_2)$.
\item The elements $u,v,y_1,y_2$ are pairwise distinct.
\item If $x_1\wedge x_2=0$, then $y_1\wedge y_2=u$ in~$B_{\set{C_1,C_2}}$.
\item If $x_1\vee x_2=1$, then $y_1\vee y_2 = v$ in~$B_{\set{C_1,C_2}}$.
\item If $x_1< x_2$ then $y_1\wedge y_2=y_1$ in~$B_{\set{C_1,C_2}}$; that is, $y_1<y_2$ is a chain of~$\bB_{\set{C_1,C_2}}$.
\item Assume that $x_1< x_2$ and $D=\set{0,x_1,x_2,1}\in\cC$. Let $u<y_1'<y_2'<v$ be a direct congruence chain of $\bB_{\set{D}}$. Then $y_1=y_1'$ and $y_2=y_2'$.
\end{enumerate}
\end{lemma}

\begin{proof}
We can assume that $\CG\circ\vec\bB=\Conc\circ\vec A$ and thus that $\vec\xi$ is the identity. Put $\delta_R=\delta_{\bB_R}$ for each $R\in\IC(\cC)$. Put $P=\set{C_1,C_2}$.

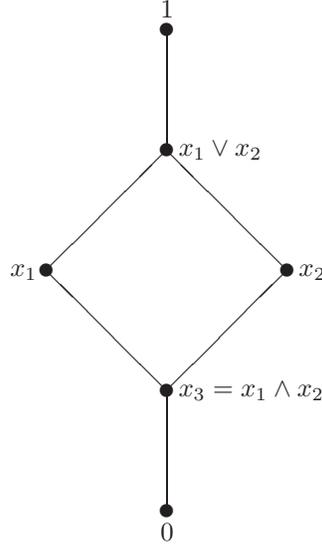
\begin{figure}[here,top,bottom]\caption{The free bounded lattice generated by $x_1$ and $x_2$.}
\setlength{\unitlength}{0.8mm}
\begin{picture}(50,90)(-5,-5)
\put(20,20){\line(1,1){20}}
\put(20,20){\line(-1,1){20}}
\put(40,40){\line(-1,1){20}}
\put(0,40){\line(1,1){20}}
\put(20,60){\line(0,1){20}}
\put(20,0){\line(0,1){20}}

\put(20,20){\circle*{2}}
\put(40,40){\circle*{2}}
\put(0,40){\circle*{2}}
\put(20,60){\circle*{2}}
\put(20,80){\circle*{2}}
\put(20,0){\circle*{2}}
\put(-6,39){$x_1$}
\put(42,39){$x_2$}
\put(22,59){$x_1\vee x_2$}

\put(19,-5){0}
\put(22,19){$x_3=x_1\wedge x_2$}
\put(19,82){1}
\end{picture}
\label{F:lat-AQ}
\end{figure}

As $u<y_k<v$ is a direct congruence chain of $\bB_{\set{C_k}}$ and $\set{C_k}\subseteq P$, the following equalities hold:
\begin{equation} \label{E:cong0x_k}
\delta_P(u,y_k) = \Theta_{A_{P}} (0,x_k),\quad\text{for all $k\in\set{1,2}$.}
\end{equation}
\begin{equation}\label{E:congx_k1}
\delta_P(y_k,v) = \Theta_{A_{P}} (x_k,1),\quad\text{for all $k\in\set{1,2}$.}
\end{equation}

Notice that $\Theta_{A_P}(0,x_1)\cap\Theta_{A_P}(x_1,1)=\zero_{A_P}$, thus:
\begin{align*}
\delta_P(y_1,y_2)&=\delta_P(y_1,y_2)\vee(\Theta_{A_P}(0,x_1)\cap\Theta_{A_P}(x_1,1))\\
&=\delta_P(y_1,y_2)\vee(\delta_P(u,y_1)\cap\delta_P(y_1,v)) && \text{by~\eqref{E:cong0x_k} and~\eqref{E:congx_k1}.}
\end{align*}
Therefore, as $\Conc A_P$ is distributive, we obtain
\begin{equation}
\delta_P(y_1,y_2)=(\delta_P(y_1,y_2)\vee\delta_P(u,y_1))\cap(\delta_P(y_1,y_2)\vee\delta_P(y_2,v)).
\end{equation}
Moreover, the following equalities hold:
\begin{align}
\delta_P(y_1,y_2)\vee\delta_P(u,y_1)&=\delta_P(u,y_1)\vee\delta_P(u,y_2)\notag\\
&=\Theta_{A_P}(0,x_1)\vee\Theta_{A_P}(0,x_2) &&\text{by~\eqref{E:cong0x_k}.}\notag\\
&=\Theta_{A_P}(0,x_1\vee x_2) &&\text{see Figure~\ref{F:lat-AQ}.}\label{E:congmajor1}
\end{align}
With a similar argument we obtain
\begin{equation}\label{E:congmajor2}
\delta_P(y_1,y_2)\vee\delta_P(y_2,v)=\Theta_{A_P}(x_1\wedge x_2,1).
\end{equation}
It follows from~\eqref{E:congmajor1} and~\eqref{E:congmajor2} that:
\[
\delta_P(y_1,y_2)=\Theta_{A_P}(0,x_1\vee x_2)\cap \Theta_{A_P}(x_1\wedge x_2,1)=\Theta_{A_P}(x_1,x_2).
\]

As $u<y_k<v$ is a congruence chain, it follows that $u\not=v$, $u\not=y_k$, and $y_k\not=v$ for all $k\in\set{1,2}$. Moreover, as $\delta_P(y_1,y_2)=\Theta_{A_P}(x_1,x_2)\not=\zero_{A_P}$, we get $y_1\not=y_2$.

Assume that $x_1\wedge x_2=0$, so the lattice~$A_{P}$ is (a quotient of) the lattice of Figure~\ref{F:lat-AQ}.

As $u<y_k<v$ is a congruence chain of $\bB_{\set{C_k}}$, we get that $u\wedge y_k=u$ in~$B_P$ for each $k\in\set{1,2}$. As $y_1,y_2\in B_P^*$, $y_1\wedge y_2$ is defined in~$B_P$. Thus~\eqref{E:cong0x_k} implies the following containments:
\[
\delta_P(u,y_1\wedge y_2)\subseteq \delta_P(u,y_1)\cap \delta_P(u,y_2)= \Theta_{A_{P}} (0,x_1)\cap \Theta_{A_{P}}(0,x_2)
= \zero_{A_{P}}
\]
and so $y_1\wedge y_2=u$ in~$B_P$. Similarly, if $x_1\vee x_2=1$, using Lemma~\ref{L:presqueordre} and~\eqref{E:congx_k1}, we obtain $y_1\vee y_2=v$ in~$B_P$.

Now assume that $x_1<x_2$. So the lattice~$A_P$ is a chain. Moreover, as $u,y_1,y_2,v$ belong to~$B_P^*$ and~$\bB_P$ is a strong gamp, the elements $y_1\wedge y_2$, $u\wedge y_1=u$, $u\wedge y_2=u$, and $v\wedge y_1=y_1$ are defined in~$B_P$. Thus $\delta_P(y_1\wedge y_2,y_1)\subseteq \delta_P(u,y_1\wedge y_2)\vee\delta_P(u,y_1)=\delta_P(u,y_1)$. Hence~\eqref{E:cong0x_k} and~\eqref{E:congx_k1} implies the following equalities:
\[
\delta_P(y_1\wedge y_2,y_1) \subseteq \delta_P(u,y_1)\cap \delta_P(y_2,v)= \Theta_{A_{P}} (0,x_1)\cap \Theta_{A_{P}}(x_2,1)=\zero_{A_{P}}
\]
therefore $y_1\wedge y_2=y_1$ in~$B_P$.

Put $D=\set{0,x_1,x_2,1}$, assume that $D\in\cC$, let $u<y_1'<y_2'<v$ be a direct congruence chain of $\bB_{\set{D}}$. Let $k\in\set{1,2}$. The following containment holds:
\[
\delta_{\set{C_k,D}}(y_k\wedge y'_k,y_k)\subseteq \delta_{\set{C_k,D}}(y_k\wedge y'_k,u)\vee \delta_{\set{C_k,D}}(u,y_k).
\]
Moreover $\delta_{\set{C_k,D}}(y_k\wedge y'_k,u)=\delta_{\set{C_k,D}}(y_k\wedge y'_k,u\wedge y'_k)\subseteq\delta_{\set{C_k,D}}(u,y_k)$. Therefore:
\begin{equation}\label{E:ykwykp1}
\delta_{\set{C_k,D}}(y_k\wedge y'_k,y_k)\subseteq \delta_{\set{C_k,D}}(u,y_k).
\end{equation}
The following containment also holds:
\begin{equation}\label{E:ykwykp2}
\delta_{\set{C_k,D}}(y_k\wedge y'_k,y_k)=\delta_{\set{C_k,D}}(y_k\wedge y'_k,y_k\wedge v)\subseteq \delta_{\set{C_k,D}}(y'_k,v).
\end{equation}
The containments \eqref{E:ykwykp1} and \eqref{E:ykwykp2} imply:
\begin{align*}
\delta_{\set{C_k,D}}(y_k\wedge y'_k,y_k)&\subseteq \delta_{\set{C_k,D}}(u,y_k)\cap \delta_{\set{C_k,D}}(y'_k,v)\\
&=\Theta_{A_{\set{C_k,D}}}(0,x_k) \cap \Theta_{A_{\set{C_k,D}}}(x_k,1)\\
&=\zero_{A_{\set{C_k,D}}}.
\end{align*}
Thus $y_k=y_k\wedge y_k'$. Similarly $y'_k=y_k\wedge y_k'$, thus $y_k=y'_k$.
\end{proof}

Given a partial lifting $\vec \bB$ of $\Conc\circ\vec A$ with enough direct congruence chains, where~$\vec A$ is the chain diagram of some partial sublattice~$K$ of~$L$, we can now construct a partial lattice in $\vec \bB$ isomorphic to~$K$.

\begin{theorem}\label{T:good-liftingdia-imply-lifting}
Let~$L$ be a bounded lattice and let~$K$ be a spanning partial sublattice of~$L$. Let~$\vec A$ be the chain diagram of~$K$ in~$L$ \pup{that is, the~$\cC_K$-chain diagram of~$L$, see Definition~\textup{\ref{D:chaindia}}}, let $\vec\bB=\famm{\bB_P,\bgg_{P,Q}}{P\le Q\text{ in }\IC(\cC_K)}$ be a partial lifting and let $\vec\xi\colon\CG\circ\vec\bB\to\Conc\circ\vec A$ be a natural equivalence. Let $u,v$ in~$B_{\emptyset}$. Assume that for each~$C\in\cC_K$ there exists a direct congruence chain of $\bB_{\set{C}}$ with extremities $g_{\emptyset,\set{C}}(u)$ and $g_{\emptyset,\set{C}}(v)$.

Given $x\in K-\set{0,1}$, set~$C_x=\set{0,x,1}$ and let $g_{\emptyset,\set{C_x}}(u)<t_x<g_{\emptyset,\set{C_x}}(v)$ be a direct congruence chain of $\bB_{\set{C_x}}$. Put:
\begin{align*}
h\colon K &\to B_\top,\\
x &\mapsto g_{\set{C_x},\top}(t_x),\\
0 &\mapsto g_{\emptyset,\top}(u),\\
1 &\mapsto g_{\emptyset,\top}(v).
\end{align*}
The map $h$ is an embedding of partial lattices and $(h,\xi_{\top})\colon (K,K,\Theta_L,\Conc L)\to\bB_{\top}$ is an embedding of gamps.
\end{theorem}

\begin{remark}
In the context of Theorem~\ref{T:good-liftingdia-imply-lifting}, note that $C_x$ is a chain of $L$ contained in $\cK$, therefore $C_x\in\cC_K$.
\end{remark}

\begin{proof}
In this proof we set~$\cC=\cC_K$. As $f_{P,Q}$ is an embedding of lattices and~$\cI(\cC_K)$ has a largest element, we can assume that $g_{P,Q}$ is the inclusion map for all $P\le Q$ in $\IC(\cC)$. We can also assume that $\CG\circ\vec\bB=\Conc\circ\vec A$. Denote $\delta_P=\delta_{\bB_P}$ for each $P\in\IC(\cC)$.
With those assumptions , $h(0)=u$, $h(1)=v$ and $h(x)=t_x$ for all $x\in K-\set{0,1}$. It follows from Lemma~\ref{L:mainchaindia}(2) that the map $h$ is one-to-one.

Let $x_1, x_2\in K$. If either $x_1 = 0$, $x_1=1$, $x_2=0$, $x_2=1$, or $x_1=x_2$, then $h(x_1\wedge x_2)=h(x_1)\wedge h(x_2)$ and $h(x_1\vee x_2)=h(x_1)\vee h(x_2)$. Now assume that $x_1\not=x_2$, $x_1\not\in\set{0,1}$, and $x_2\not\in\set{0,1}$. Set~$C_k=\set{0,x_k,1}$, set $y_k=h(x_k)$. Hence $u<y_k<v$ is a direct congruence chain of $\bB_{\set{C_k}}$ for all $k\in\set{1,2}$. If $x_1\wedge x_2=0$, it follows from Lemma~\ref{L:mainchaindia}(3) that $y_1\wedge y_2=u$ in~$B_{\set{C_1,C_2}}$, so it also holds in~$B_\top$, hence $h(x_1\wedge x_2)=h(0)=u=y_1\wedge y_2=h(x_1)\wedge h(x_2)$. Similarly, if $x_1\vee x_2=1$, then $h(x_1\vee x_2)=h(x_1)\vee h(x_2)$, and, if $x_1<x_2$ then $h(x_1)\wedge h(x_2)=h(x_1)$, so $h(x_1\wedge x_2)=h(x_1)=h(x_1)\wedge h(x_2)$ and $h(x_1\vee x_2)=h(x_2)=h(x_1)\vee h(x_2)$.

Assume that $x_1$ and $x_2$ are incomparable and that $x_1\wedge x_2\in K-\set{0}$. Set $x_3=x_1\wedge x_2$. Put~$C_k=\set{0,x_k,1}$, set $y_k=h(x_k)$, hence $u<y_k<v$ is a direct congruence chain of~$\bB_{\set{C_k}}$, for all $k\in\set{1,2,3}$. Put $D_k=\set{0,x_3,x_k,1}$ for all $k\in\set{1,2}$. Let $u<y_3'<y_2'<v$ be a direct congruence chain of $\bB_{\set{D_2}}$ and let $u<y_3''<y_1''<v$ be a direct congruence chain of $\bB_{\set{D_1}}$. It follows from Lemma~\ref{L:mainchaindia}(6) that $y_3=y_3'=y_3''$, $y_2=y_2'$, and $y_1=y_1'$. Hence $u<y_3<y_k<v$ is a direct congruence chain of $\bB_{\set{D_k}}$ for all $k\in\set{1,2}$. Thus the following equalities hold:
\begin{equation}\label{Eq1:MainTh}
\delta_{\set{C_1,D_2}}(y_3,y_2)=\Theta_{A_{\set{C_1,D_2}}} (x_3,x_2),
\end{equation}
\begin{equation}\label{Eq2:MainTh}
\delta_{\set{C_1,D_2}}(u,y_1) = \Theta_{A_{\set{C_1,D_2}}}(0,x_1).
\end{equation}

As $u<y_3<y_1<v$ and $u<y_3<y_2<v$ are chains of $\bB_{\set{C_1,D_2}}$, it follows that
\begin{align*}
\delta_{\set{C_1,D_2}}(y_3,y_1\wedge y_2)&=\delta_{\set{C_1,D_2}}(y_3\wedge y_2,y_1\wedge y_2)\\
&\subseteq\delta_{\set{C_1,D_2}}(y_3,y_1)\\
&=\delta_{\set{C_1,D_2}}(u\vee y_3,y_1\vee y_3)\\
&\subseteq\delta_{\set{C_1,D_2}}(u,y_1)
\end{align*}
and $\delta_{\set{C_1,D_2}}(y_3,y_1\wedge y_2)=\delta_{\set{C_1,D_2}}(y_1\wedge y_3,y_1\wedge y_2)\subseteq \delta_{\set{C_1,D_2}}(y_3, y_2)$. Therefore the following containments hold:
\begin{align*}
\delta_{\set{C_1,D_2}}(y_3,y_1\wedge y_2) &\subseteq \delta_{\set{C_1,D_2}}(u,y_1)\cap \delta_{\set{C_1,D_2}}(y_3,y_2)\\
&=\Theta_{A_{\set{C_1,D_2}}}(0,x_1)\cap \Theta_{A_{\set{C_1,D_2}}}(x_3,x_2) & & \text{by~\eqref{Eq1:MainTh} and~\eqref{Eq2:MainTh}}\\
&=\zero_{A_{\set{C_1,D_2}}}, & & \text{see Figure~\ref{F:lat-AQ}}.
\end{align*}
Thus $y_3=y_1\wedge y_2$ in~$B_\top$, so $h(x_1\wedge x_2)=h(x_3)=y_3=y_1\wedge y_2=h(x_1)\wedge h(x_2)$. Similarly if $x_1$ and $x_2$ are incomparable and $x_1\vee x_2\in K-\set{1}$ then $h(x_1\vee x_2)=h(x_1)\vee h(x_2)$. Hence $h$ is a morphism of partial lattices from~$K$ to~$B_\top$.

Let $x,y\in K-\set{0,1}$. Lemma~\ref{L:mainchaindia}(1) implies that:
\[
\delta_{\set{C_x,C_y}}(h(x),h(y))=\delta_{\set{C_x,C_y}}(t_x,t_y)=\Theta_{A_{\set{C_x,C_y}}}(x,y).
\]
Therefore $\delta_{\top}(h(x),h(y))=\Theta_{L}(x,y)$.

Let $x\in K$. {}From $\delta_{\set{C_x}}(u,t_x)=\Theta_{C_x}(0,x)$ it follows that $\delta_{\top}(h(0),h(x))=\Theta_{L}(0,x)$. A similar argument gives $\delta_{\top}(h(x),h(1))=\Theta_{L}(x,1)$. Moreover, $\delta_\emptyset(u,v)=\Theta_{A_\emptyset}(0,1)$, which implies in turn $\delta_{\top}(f(0),f(1))=\Theta_{L}(0,1)$.

Therefore $(h,\xi_{\top})\colon (K,K,\Theta_L,\Conc L)\to\bB_{\top}$ is an embedding of gamps.
\end{proof}

Gluing the chain diagram of a partial lattice~$K$ of~$L$ and the ``directing'' diagrams constructed in Lemma~\ref{L:directing}, we obtain a result similar to Theorem~\ref{T:good-liftingdia-imply-lifting}. We still need to assume the existence of enough congruence chains but these no longer need to be direct. As our directing diagrams ``force'' all congruence chains to be either direct or dually direct, our result is stated up to dualization.

\begin{lemma}\label{L:exists-good-diagram}
Let~$\cK$ be a class of bounded lattices closed under finite products and containing all bounded lattices generated by two elements. Assume that either~$M_3$ or~$N_5$ belongs to~$\cK$. Let $L\in\cK$, let~$K$ be a spanning finite partial sublattice of~$L$, such that $K$ has at least five elements. Put $\cC=\cC_K$. There exists a direct system $\vec A=\famm{A_P,f_{P,Q}}{P\le Q\text{ in } \IC(\cC)}$ of~$\cK$ \pup{with $0,1$-lattice homomorphisms} such that:
\begin{enumerate}
\item The lattice~$A_P$ is finite and distributive, for each $P<\top$ in $\IC(\cC)$.
\item If~$L$ is finite then $A_\top$ is finite, otherwise $\card A_\top=\card L$.
\item The equality $\vec A\res \JC(\cC)=\vec E(\cC)$ holds.
\item Let $\vec \bB=\famm{\bB_P,\bgg_{P,Q}}{P\le Q\text{ in }\IC(\cC)}$ be a partial lifting of $\Conc\circ\vec A$, let $u,v$ in~$B_{\emptyset}$. If~$\bB_{\set{C}}$ contains a congruence chain with extremities $g_{\emptyset,\set{C}}(u)$ and $g_{\emptyset,\set{C}}(v)$ for each~$C\in\cC$, then there exists a subgamp of a quotient of either~$\bB_\top$ or its dual isomorphic to~$(K,K,\Theta_L,\Conc L)$.
\end{enumerate}
\end{lemma}

\begin{proof}
The $\cC$-chain diagram~$\vec A^0$ of~$L$ is a diagram in~$\cK$. 

Let~$C_1,C_2,D\in\cC$ be pairwise distinct chains such that~$C_1$ and~$C_2$ both have length~$2$, and either $D$ has length~$2$ or~$C_1,C_2\subseteq D$.

As $M_3\in\cK$ or $N_5\in\cK$, it follows from Lemma~\ref{L:directing} that there exists a diagram $\vec H^{C_1,C_2,D}$ of finite lattices of~$\cK$ indexed by $\IC(\set{C_1,C_2,D})$ such that:
\begin{enumerate}
\item The equality $\vec H^{C_1,C_2,D}\res \JC(\set{C_1,C_2,D})=\vec E(\set{C_1,C_2,D})$ holds.
\item Let $\vec\bB=\famm{\bB_P,\bgg_{P,Q}}{P\le Q\text{ in }\IC(\set{C_1,C_2,D})}$ be a partial lifting, let $\vec\xi\colon\CG\circ\vec\bB\to\Conc\circ\vec H^{C_1,C_2,D}$ be a natural equivalence, let $u,v$ in~$B_{\emptyset}$. Assume that $\bB_{\set{C_k}}$ contains a direct congruence chain with extremities $g_{\emptyset,\set{C_k}}(u)$ and $g_{\emptyset,\set{C_k}}(v)$ for $k\in\set{1,2}$. If $\bB_{\set{D}}$ contains a congruence chain with extremities $g_{\emptyset,\set{D}}(u)$ and $g_{\emptyset,\set{D}}(v)$ then it is also direct.
\end{enumerate}
It follows from Lemma~\ref{L:diagramextension} that there exists a direct system $\vec A^{C_1,C_2,D}$ of finite lattices in~$\cK$ indexed by $\IC(\cC)$, such that:
\begin{align*}
\vec A^{C_1,C_2,D}\res \IC(\set{C_1,C_2,D})&=\vec H^{C_1,C_2,D},\\
\vec A^{C_1,C_2,D}\res\JC(\cC)&=\vec E(\cC).
\end{align*}
Let $T$ be the set with elements~$0$ and $(C_1,C_2,D)$ where~$C_1,C_2,D\in\cC$ are pairwise distinct and either $D$ has length~$2$ or~$C_1,C_2\subseteq D$.
Let $\vec A=\famm{A_P,f_{P,Q}}{P\le Q\text{ in }\IC(\cC)}$ be the product of $(\vec A^t)_{t\in T}$ over $\JC(\cC)$ (cf. Definition~\ref{D:productdiagram}) and let $\vec\pi^t=(\pi_P^t)_{P\in \IC(\cC)}\colon \vec A\to  \vec A^t$ be the canonical projection for each $t\in T$ (see Definition~\ref{D:productdiagram}).

It follows from Remark~\ref{R:productdiagram} that~$\vec A$ is a diagram in~$\cK$ and the lattice~$A_P$ is finite for each $P<\top$ in $\IC(\cC)$. The lattice~$A_\top$ is a finite product of finite lattices and~$L$, so $(2)$ holds. {}From the definition of the product over $\JC(\cC)$ (cf. Definition~\ref{D:productdiagram}), it follows that $\vec A\res \JC(\cC)=\vec A^0\res \JC(\cC)=\vec E(\cC)$.

Given $t\in T$, we put $\vec I^t=\ker_0\vec\pi^t$, that is
\[
I_P^t=\ker_0\GA(\pi_P^t)=\ker_0(\Conc\pi_P^t)=(\Conc A_P)\dnw\ker\pi_P^t,\quad\text{for each $P\in\IC(\cC)$.}
\]
Let $\vec\chi^t\colon (\Conc \circ\vec A)/\vec I^t\to\Conc\circ\vec A^t$ induced by $\Conc\circ\vec\pi^t$. It follows from \cite[Lemma~3.13]{G3} that $\vec\chi^t$ is a natural equivalence. Moreover $I_P^t=\set{0}$,~$A_P=A_P^t$, and $\chi_P^t=\id_{\Conc A_P}$ for each $P\in\JC(\cC)$ and each $t\in T$.

Let $\vec\bB=\famm{\bB_P,\bgg_{P,Q}}{P\le Q\text{ in }\IC(\cC)}$ be a partial lifting of $\Conc\circ\vec A$. We can assume that $\CG\circ\vec\bB=\Conc\circ\vec A$. Let $u,v\in B_{\emptyset}$. Suppose that $\bB_{\set{C}}$ contains a congruence chain with extremities $g_{\emptyset,\set{C}}(u)$ and $g_{\emptyset,\set{C}}(v)$ for each~$C\in\cC$. In the rest of our proof, all congruence chains of~$\bB_{\set{C}}$ will have extremities $g_{\emptyset,\set{C}}(u)$ and $g_{\emptyset,\set{C}}(v)$.

\begin{sclaim}
Let~$C_1,C_2\in\cC$ be distinct chains of length~$2$, let $D\in\cC-\set{C_1,C_2}$ such that either $D$ has length~$2$ or~$C_1,C_2\subseteq D$. Assume that $\bB_{\set{C_k}}$ contains a direct congruence chain for $k=1,2$. Then $\bB_{\set{D}}$ contains a direct congruence chain.
\end{sclaim}

\begin{scproof}
Put $t=(C_1,C_2,D)$. The restriction $(\vec\bB/\vec I^t)\res\IC(\set{C_1,C_2,D})$ is a partial lifting of $\Conc\circ\vec H^t$ induced by the restriction of $\vec\chi^t$. As $\chi_P^t=\id_{\Conc A_P}$ and $I_P^t=\set{0}$, it follows that the direct congruence chains of $\bB_P/I_P^t$ and $\bB_P$ are the same for both structures for each $P\in \set{\set{C_1},\set{C_2},\set{D}}$. It follows from~$(2)$ that every congruence chain of~$\bB_{\set{D}}$ is direct.
\end{scproof}

Let~$C$ be a chain of length~$2$. A congruence chain of $\bB_{\set C}$ is either direct or dually direct (cf. Remark~\ref{R:congruencechain}). As $K$ has at least five elements, there are at least three chains of length~$2$ in $\cC$. Therefore, up to changing~$\vec\bB$ to its dual, we can take~$C_1,C_2\in\cC$ distinct of length~$2$, such that $\bB_{\set{C_k}}$ has a direct congruence chain. Let $D\in\IC(\cC)-\set{C_1,C_2}$ be a chain of length~$2$. It follows from the Claim that all congruence chains of $\bB_{\set D}$ are direct. Therefore $\bB_{\set D}$ has a direct congruence chain for every $D\in\IC(\cC)$ of length~$2$.

Let $D=\set{0,x_1,x_2,1}\in\IC(\cC)$ be a chain of length~$3$, put~$C_k=\set{0,x_k,1}$ for $k=1,2$. As $\bB_{\set{C_k}}$ has a direct congruence chain for $k=1,2$, it follows from the Claim that $\bB_{\set D}$ has a direct congruence chain.

Thus, applying Theorem~\ref{T:good-liftingdia-imply-lifting}, to $\vec\bB/\vec I^0$ which is a partial lifting of $\Conc\circ\vec A^0$, we obtain a subgamp of a~$\bB_\top/I^0_\top$ isomorphic to~$(K,K,\Theta_L,\Conc L)$.
\end{proof}

The following lemma gives a way to find a partial sublattice of some lattice in a variety but not in another.

\begin{lemma}\label{L:exists-good-partial-sublattice}
Let~$\cV$ be a variety of bounded lattices, let~$\cW$ be variety of lattices. If $\cV\not\subseteq\cW$ and $\cV\not\subseteq\dual{\cW}$ then there are a countable bounded lattice $L\in\cV$ and a finite spanning partial sublattice~$K$ of~$L$ such that~$K$ is not a partial sublattice of any lattice of $\cW\cup\dual{\cW}$.
\end{lemma}

\begin{proof}
Assume that $\cV\not\subseteq\cW$ and $\cV\not\subseteq\dual{\cW}$. Let $t_1=t_2$ be an identity satisfied in~$\cW$ but not satisfied in~$\cV$, let $t_1'=t_2'$ be an identity satisfied in~$\dual{\cW}$ but not in~$\cV$. Let $L\in\cV$ be a countable bounded lattice that fails both $t_1=t_2$ and $t'_1=t'_2$, let~$K$ be a finite spanning partial sublattice of~$L$ which fails both $t_1=t_2$ and $t'_1=t'_2$ (cf. \cite[Definition 4.8]{G3}). As~$K$ does not satisfy $t_1=t_2$, it is not a partial sublattice of any lattice of~$\cW$. Similarly~$K$ it is not a partial sublattice of any lattice of $\dual{\cW}$.
\end{proof}

Lemma~\ref{L:exists-good-partial-sublattice} and Lemma~\ref{L:exists-good-diagram} are the mains tools used in this paper to construct a diagram liftable in a variety and not in another one.

\section{A larger diagram}\label{S:ALargerDiag}

The aim of this section is, given a diagram~$\vec A$ indexed by a poset $I$, to construct a new diagram $\vec A'$ (cf. Lemma~\ref{L:complexdiagram}) indexed by a larger poset~$J$ (cf. \cite[Definition~8.6]{G3}), such that the existence of a partial lifting of $\Conc\circ\vec A'$ in some ``good'' variety implies the existence of a partial lifting of $\Conc\circ\vec A$ with many congruence chains (cf. Lemma~\ref{L:complextosimplewithchain}).

The proof of the following lemma is straightforward. We refer to \cite[Definition~8.6, Remark~8.7]{G3} for the definition of $\kposet{I}{X}{\vec P}{\alpha}$ and its \emph{associated tree}.

\begin{lemma}\label{L:complexdiagram}
Let~$\cK$ be a class of bounded lattices, let~$I$ be a poset with a smallest element~$0$, let $\vec A=(A_i,f_{i,j})_{i\le j\text{ in }I}$ be a direct system of~$\cK$ such that~$A_0=\two$, let $X\subseteq I-\set{0}$ such that~$A_x$ is a finite chain of length at least~$2$, for each $x\in X$, and let $\alpha\le\omega$ be an ordinal. Put
 \[
 P_x=\setm{p\colon A_x\tosurj \two}{\text{$p$ is isotone and preserves bounds}},
 \quad\text{for each }x\in X.
 \]
Set $\vec P=(P_x)_{x\in X}$ and $J=\kposet{I}{X}{\vec P}{\alpha}$, and denote by $T$ its associated tree. Put:
\[
A'_{(t,i)}=A_i,\quad\text{for each $(t,i)\in J$.}
\]

Let $t=(n,\vec x,\vec p)\in T$, let $\ba=(t\res m,i)\le\bb=(t,j)$ in $J$. If $m=n$ we put $f'_{\ba,\bb}=f_{i,j}$. If $m<n$ we put $f'_{\ba,\bb}=f_{0,j}\circ p_{m}\circ f_{i,x_{m}}$.
Then $\vec A'=\famm{A'_{\ba},f'_{\ba,\bb}}{\ba\le\bb\text{ in }J}$ is a $J$-indexed diagram in~$\cK$.
\end{lemma}

\begin{lemma}\label{L:complextosimplewithchain}
We use the notation of Lemma~\textup{\ref{L:complexdiagram}}, with $\alpha=\omega$. Let~$\cV$ be a variety of lattices such that every simple lattice in~$\cV$ contains a prime interval. Let $\famm{\bB_{\ba},\bgg_{\ba,\bb}}{\ba\le \bb\text{ in }J}$ be a partial lifting of~$\Conc\circ\vec A'$ such that $\bB_{\ba}$ is finite for each $\ba\in J^=$. Then there are $t\in T$ and $u,v$ in~$B_{(t,0)}$ such that for each $x\in X$, the gamp $\bB_{(t,x)}$ has a congruence chain with extremities $g_{(t,0),(t,x)}(u)$ and $g_{(t,0),(t,x)}(v)$.
\end{lemma}

\begin{proof}
Denote by $T$ the tree associated to $\kposet{I}{X}{\vec P}{\alpha}$. Let $t=(n,\vec x,\vec p)\in T$; then $J_t=\setm{(t,i)}{i\in I}$ is a subposet of~$J$. Moreover, the assignment $I\to J_t$, $i\mapsto(t,i)$ is an isomorphism, and it induces an isomorphism of the diagrams $\vec A'\res J_t$ and~$\vec A$.

Assume that for all $t\in T$ and for every chain $u<v$ of $\bB_{(t,0)}$, there exists $x\in X$ such that $\bB_{(t,x)}$ has no congruence chain with extremities $g_{(t,0),(t,x)}(u)$ and $g_{(t,0),(t,x)}(v)$.

Our aim is to construct a sequence $\vec x=(x_k)_{k<\omega}$ of $X$, $\vec p\in\prod_{k<\omega}P_{x_k}$, and $\ba_n=(n,\vec x\res n,\vec p\res n,0)$ for each $n<\omega$, such that for each $m<\omega$ and for each chain $u<v$ in $\bB_{\ba_m}$, there are $n>m$ and $z\in B_{\ba_n}^*$ such that $g_{\ba_m,\ba_n}(u)<z<g_{\ba_m,\ba_n}(v)$ is a chain of $\bB_{\ba_n}$.

We can assume that $\CG\circ\vec B=\Conc\circ \vec A'$ (cf. \cite[Remark~6.15]{G3}). We set $\delta_{\ba}=\delta_{\bB_{\ba}}$, for each $\ba\in J$.

Assume having already constructed $\vec x\in X^n$ and $\vec p\in P_{x_0}\times\dots\times P_{x_{n-1}}$ for some $n<\omega$. Set $\ba_{m}=(m,\vec x\res m,\vec p\res m,0)$ for each $m\le n$. As~$A_{\ba_{n}}'=A_0=\two$ there exists a non-zero congruence of~$A_{\ba_{n}}'$. Moreover, $\bB_{\ba_n}$ is distance-generated with chains, thus there exist a chain $u<v$ of $\bB_{\ba_{n}}$. Using the finiteness of $\bB_{\ba_{n}}$ we can construct a covering $u'\prec v'$ in $\bB_{\ba_{n}}$.

Let $m\le n$ minimal such that there exist $u\prec v$ in $\bB_{\ba_m}$ with $g_{\ba_m,\ba_{n}}(u)\prec g_{\ba_m,\ba_{n}}(v)$. Let $u,v$ be such elements, thus $u\not=v$. Note that $\Conc f_{\ba_m,\ba_{n}}$ is the identity map, hence it separates zero; it follows from \cite[Proposition 5.10(2)]{G3} that $g_{\ba_m,\ba_{n}}(u)\not=g_{\ba_m,\ba_{n}}(v)$. Put $u'=g_{\ba_m,\ba_{n}}(u)$, and $v'=g_{\ba_m,\ba_n}(v)$. So $\delta_{\ba_n}(u',v')\not=\zero_{A_{\ba_n}}$, however~$A_{\ba_n}=\two$, therefore $\delta_{\ba_n}(u',v')=\one_{A_{\ba_n}}$.

Put $t=(n,\vec x,\vec p)$. Let $y\in X$ such that $\bB_{(t,y)}$ has no congruence chain with extremities $g_{\ba_n,(t,y)}(u')$ and $g_{\ba_n,(t,y)}(v')$. Set $\by=(t,y)$. Put $S=\At\Conc A_{\by}$, we recall that~$A_{\by}=A_y$ is a chain. So the following equalities hold:
\[
\delta_{\by}(g_{\ba_n,\by}(u'),g_{\ba_n,\by}(v'))=(\Conc f_{0,y})(\delta_{\ba_n}(u',v'))=(\Conc f_{0,y})(\one_{A_0})=\one_{A_y}=\bigvee S.
\]

As $\bgg_{\ba_n,\by}$ is congruence-cuttable with chains, there exists a chain $g_{\ba_n,\by}(u')=u_0<u_1<\dots <u_{\ell}=g_{\ba_n,\by}(v')$ of $\bB_{\by}$ such that for each $k<\ell$ there exists $\alpha\in S$ such that $\delta_{\by}(u_k,u_{k+1})\le \alpha$; moreover, as $\alpha$ is an atom, $\delta_{\by}(u_k,u_{k+1})=\alpha$.

Let $\alpha\in S$. If there does not exist $k<\ell$ such that $\alpha=\delta_{\by}(u_k,u_{k+1})$, then:
\[
\bigvee(S-\set{\alpha})\ge\bigvee_{k<\ell}\delta_{\by}(u_k,u_{k+1})\ge\delta_{\by}(u_0,u_{\ell})=\one_{A_y}\ge\alpha,
\]
a contradiction as~$S$ is the set of atoms of the Boolean lattice $\Conc A_y$. Hence for each $\alpha\in S$ there exists $k<\ell$ such that $\alpha=\delta_{\by}(u_k,u_{k+1})$. However, $g_{\ba_n,\by}(u')=u_0<u_1<\dots<u_{\ell}=g_{\ba_n,\by}(v')$ is not a congruence chain of~$\bB_{\by}$, so there are $k<k'<\ell$ and $\alpha\in S$ such that $\alpha=\delta_{\by}(u_k,u_{k+1})=\delta_{\by}(u_{k'},u_{k'+1})$. Let $p_{n}\colon A_y\tosurj A_0=\two$ be isotone such that $(\Conc p_{n})(\alpha)=\one_{A_0}$.

Put $x_{n}=y$, $\ba_{n+1}=(n+1,(\vec x,x_n),(\vec p,p_n),0)$, $u''=g_{\ba_m,\ba_{n+1}}(u)$, $v''=g_{\ba_m,\ba_{n+1}}(v)$, and $z=g_{\by,\ba_{n+1}}(u_{k+1})$. Notice that $u''=g_{\ba_m,\ba_{n+1}}(u)=g_{\by,\ba_{n+1}}(u_0)$,  $f_{\by,\ba_{n+1}}'=p_{n+1}$, and $\delta_{\ba_{n+1}}(u_0,u_{k+1})\ge\alpha$, thus
\begin{align*}
\delta_{\ba_{n+1}}(u'',z)&=\delta_{\ba_{n+1}}(g_{\by,\ba_{n+1}}(u_0),g_{\by,\ba_{n+1}}(u_{k+1}))\\
&= (\Conc f_{\by,\ba_{n+1}}')(\delta_{\ba_{n+1}}(u_0,u_{k+1}))\\
&\ge(\Conc p_{n+1})(\alpha)\\
&=\one_{A_0}.
\end{align*}
Hence $z\not=u''$. Similarly $z\not=v''$, so $u''=g_{\ba_m,\ba_{n+1}}(u)<z<g_{\ba_m,\ba_{n+1}}(v)=v''$ is a chain of $\bB_{\ba_{n+1}}$.

Arguing by induction, we thus construct a subdiagram
 \[
 \famm{\bB_{\ba_m},\bgg_{\ba_m,\ba_{n}}}{m\le n<\omega}
 \]
of~$\vec\bB$. It follows from the construction that for all $m<\omega$ and all $u\prec v$ in $\bB_{\ba_m}$ there are $n>m$ and a chain $g_{\ba_m,\ba_n}(u)<z<g_{\ba_m,\ba_n}(v)$ of $\bB_{\ba_n}$.

Let $\vec D=(\two,\id_\two)_{m\le n<\omega}$. It is easy to check that $\famm{\bB_{\ba_m},\bgg_{\ba_m,\ba_{n}}}{m\le n<\omega}$ is a partial lifting of $\vec D$. Hence it follows from \cite[Lemma~6.17]{G3} that
 \[
 \bB=\varinjlim \famm{\bB_{\ba_m},\bgg_{\ba_m,\ba_{n}}}{m\le n<\omega}
 \]
is a lattice in~$\cV$ and $\Conc B\cong \two$, that is, $B$ is simple. By construction, there does not exist $u\prec v$ in~$B$; a contradiction.
\end{proof}

\begin{theorem}\label{T:premajorcritpoint}
Let~$\cK$ be a class of bounded lattices closed under finite products and directed colimits such that either $M_3\in\cK$ or $N_5\in\cK$ and every bounded lattice generated by two elements belongs to~$\cK$. Let $L\in\cK$ and let~$K$ be a finite partial sublattice of~$L$. There exists a lattice $A\in\cK$ such that the following statements hold:
\begin{itemize}
\item $\card A=\aleph_2+\card L$.
\item Let~$B$ be a lattice such that every simple lattice in $\Var B$ contains a prime interval. If $\Conc B\cong\Conc A$, then $K$ is a partial sublattice of a quotient of either~$B$ or its dual.
\end{itemize}
\end{theorem}

\begin{proof}
Denote $\cC=\cC_K$. It follows from Lemma~\ref{L:exists-good-diagram} that there exists a direct system $\vec A=(A_P,f_{P,Q})_{P\le Q\text{ in } \IC(\cC)}$ of~$\cK$ (with $0,1$-lattice homomorphisms) such that:
\begin{enumerate}
\item The lattice~$A_P$ is finite, for each $P<\top$ in $\IC(\cC)$.
\item $\card A_\top\le\card L+\aleph_0$.
\item The equality $\vec A\res \JC(\cC)=\vec E(\cC)$ holds.
\item Let $\famm{\bB_P,\bgg_{P,Q}}{P\le Q\text{ in }\IC(\cC)}$ be a partial lifting of $\Conc\circ\vec A$, let $u,v$ in~$B_{\emptyset}$. If $\bB_{\set{C}}$ contains a congruence chain with extremities $g_{\emptyset,\set{C}}(u)$ and $g_{\emptyset,\set{C}}(v)$ for each~$C\in\cC$, then there exists a one-to-one morphism of partial lattices from~$K$ to a quotient of either~$B_\top$ or its dual.
\end{enumerate}
Put $I=\IC(\cC)$, set $X=\setm{\set{C}}{C\in\cC}$, put $P_{\set C}=\setm{p\colon C\tosurj \two}{p\text{ is isotone}}$, for each $C\in\cC$, as in Lemma~\ref{L:complexdiagram}. Put $J=\kposet{I}{X}{\vec P}{\omega}$. Let~$\vec A'$ as in Lemma~\ref{L:complexdiagram}. Then~$A'_{\ba}$ is finite for each $\ba\in J^=$ and~$\card A'_{\ba}\le\card L+\aleph_0$ for each $\ba\in \Max J$ (note that the maximal elements of $J$ are the elements of the form $(t,\top)$, where $t\in T$).

The definition of $\aleph_0$-lifter appears in \cite{GiWe1}, see also \cite[Definition~8.3]{G3}. It follows from \cite[Corollary~8.10]{G3} that there exists an $\aleph_0$-lifter $(U,\bU)$ of $J$ such that $\card U=\aleph_2$. Put~$A=\xF(U)\otimes\vec A'$ (cf. \cite[Definition 3.1-5]{GiWe1}, or \cite[Remark~9.2]{G3} for a short description). The following inequality holds:
\[
\card A\le\card U+\sum_{j\in J}\card A'_{j}=\aleph_2+\card L.
\]
Let~$B$ be a lattice such that $\Conc B\cong \Conc A=\Conc\bigl(\xF(U)\otimes\vec A'\bigr)$ and every simple lattice in $\cW=\Var B$ contains a prime interval. Thus it follows from \cite[Theorem~9.3 and Remark~9.4]{G3} that there exists $\vec\bB=\famm{\bB_{\ba},\bgg_{\ba,\bb}}{\ba\le\bb\text{ in }J}$ a (lattice) partial lifting of $\Conc\circ\vec A'$, such that $\bB_{\ba}$ is finite for each $\ba\in J^=$, and $\bB_{\ba}$ is a quotient of $\GA(B)$ for each $\ba\in \Max J$.

Now Lemma~\ref{L:complextosimplewithchain} implies that there exists $(\bb_i)_{i\in I}$ in~$J$ such that the diagram $\famm{\bB_{\bb_i},\bgg_{\bb_i,\bb_j}}{i\le j\text{ in }I}$ is a partial lifting of $\Conc\circ\vec A$, and there exists a chain $u<v$ in $\bB_{\bb_0}$ such that $\bB_{\bb_x}$ has a congruence chain with extremities $g_{\bb_0,\bb_x}(u)$ and $g_{\bb_0,\bb_x}(v)$ for each $x\in X$. So it follows from $(4)$ that~$K$ is a partial sublattice of a quotient of either~$\bB_{\bb_\top}$ or its dual. Notice that $\bb_{\top}$ is maximal in $J$, so $\bB_{\bb_\top}$ is a quotient of $\GA(B)$.
\end{proof}

\begin{corollary}
Let~$\cV$ be a variety of bounded lattices and let~$\cW$ be a variety of lattices such that every simple lattice in~$\cW$ contains a prime interval. Let $n\ge 4$, let~$L$ be a congruence $n$-permutable lattice in~$\cV$. If $L\not\in\cW\cup\dual{\cW}$, then there exists a congruence $n$-permutable lattice $A\in\cV$ such that $\Conc A\not\cong\Conc B$ for any $B\in\cW$.
\end{corollary}

\begin{proof}
With a proof similar to the one of Lemma~\ref{L:exists-good-partial-sublattice}, we can find a finite spanning partial sublattice $K$ of~$L$ such that $K$ does not embed into any lattice of $\cW\cup\dual{\cW}$. Changing~$L$ to one of its congruence $n$-permutable sublattices that contains $K$, we can assume that~$L$ is countable.

Denote by~$\cK$ the class of all congruence $n$-permutable lattices in~$\cV$. Notice that if~$\cV$ contains neither $M_3$ nor $N_5$, then~$\cV$ is distributive, so the only possibility for~$\cW$ is the trivial variety and the result holds in that case. Thus we can assume that either $M_3\in\cK$ or $N_5\in\cK$. Moreover, as $n\ge 4$, every bounded lattice generated by two elements is congruence $n$-permutable and so it belongs to~$\cK$. Theorem~\ref{T:premajorcritpoint} implies that there exists $A\in\cK$ such that $\card A=\aleph_2$ and $\Conc A$ has no lifting in~$\cW$.
\end{proof}

\begin{theorem}\label{T:majorcritpoint}
Let~$\cV$ be a variety of bounded lattices and let~$\cW$ be a variety of lattices. If every simple lattice in~$\cW$ contains a prime interval, then one of the following statements holds:
\begin{enumerate}
\item $\crit{\cV}{\cW}\le\aleph_2$.
\item $\cV\subseteq\cW$.
\item $\cV\subseteq\dual{\cW}$.
\end{enumerate}
\end{theorem}

\begin{proof}
We can assume that $M_3\in\cV$ or $N_5\in\cV$ (the result is trivial otherwise). Assume that $\cV\not\subseteq\cW$ and $\cV\not\subseteq\dual{\cW}$. It follows from Lemma~\ref{L:exists-good-partial-sublattice} that there are a countable lattice $L\in\cV$ and a finite spanning partial sublattice~$K$ of~$L$ such that~$K$ is not a partial sublattice of any lattice of $\cW\cup\dual{\cW}$. Let $A$ be the lattice constructed in Theorem~\ref{T:premajorcritpoint}, so $\card A=\aleph_2$ and $\Conc A$ cannot have a lifting in~$\cW$, hence $\crit{\cV}{\cW}\le\aleph_2$.
\end{proof}

\begin{remark}
In the proof of Theorem~\ref{T:majorcritpoint}, if $\cV\not\subseteq\cW$ and $\cV\not\subseteq\dual{\cW}$, we construct a diagram $\vec A'$ in~$\cV$ indexed by a poset~$J$, such that $\Conc\circ\vec A'$ has no partial lifting in~$\cW$, in particular it has no lifting in~$\cW$. Moreover the poset $I$ is lower finite and~$A'_i$ is finite for each $i\in I^=$ and countable otherwise.

If we assume that simple lattices of~$\cW$ have length bounded by some integer~$\alpha$, then we can ``cut'' the diagram $\vec A'$, taking $J=\kposet{I}{X}{\vec P}{\alpha}$. Thus we obtain a finite diagram $\vec A'$ such that $\Conc\circ\vec A'$ has no lifting in~$\cW$. If we also assume that~$\cV$ is finitely generated, then $\vec A'$ is a finite diagram of finite lattices, because we can choose the partial lattice~$K$ as a sublattice of a finite lattice~$L$ in Lemma~\ref{L:exists-good-partial-sublattice}.
\end{remark}

This partially answers \cite[Problem 5]{CLPSurv}, in particular for finitely generated varieties of lattices.

\begin{corollary}
Let~$\cV$ and~$\cW$ be varieties of lattices such that every simple lattice in~$\cW$ contains a prime interval \pup{this holds, in particular, in case~$\cW$ is finitely generated}. Then either $\crit{\cV}{\cW}\le\aleph_2$ or $\Conc\cV\subseteq\Conc\cW$.
\end{corollary}

\begin{proof}
Denote by $\cV^{0,1}$ the variety of all bounded lattices in $\cV$. If $\crit{\cV^{0,1}}{\cW}\le\aleph_2$, then $\crit{\cV}{\cW}\le\aleph_2$. Otherwise it follows from Theorem~\ref{T:majorcritpoint} that either $\cV^{0,1}\subseteq\cW$, or $\cV^{0,1}\subseteq\dual{\cW}$. Changing $\cW$ to its dual, we can assume that $\cV^{0,1}\subseteq\cW$.

Let $L\in\cV$, we consider $L'=L\sqcup\set{0,1}$ where~$0$ is the smallest element of $L'$ and~$1$ is the largest element of~$L'$. It is well-known that $L'$ satisfies the same identities as~$L$, moreover $L'\in\cV^{0,1}\subseteq\cW$, however $L\subseteq L'$, therefore $L\in\cW$.
\end{proof}

This also solves \cite[Problem 6]{CLPSurv}. Notice that this question is misformulated: indeed, denoting by $K$ the chain of length~$2$ and by~$L$ the chain of length~$3$, then $\Conc(\Var K)=\Conc(\Var L)$, but $K\not\cong L$ and $K\not\cong \dual{L}$. However, the following statement holds.

\begin{corollary}\label{C:solvepb6}
Let~$K$ and~$L$ be finite subdirectly irreducible lattices. If $\Conc(\Var K)=\Conc(\Var L)$, then either $K\cong L$ or $K\cong\dual{L}$.
\end{corollary}

\begin{proof}
It follows from Theorem~\ref{T:majorcritpoint} that, up to changing~$L$ to its dual, $\Var K=\Var L$. As $K$ and~$L$ are both finite and subdirectly irreducible, it follows from J\'onsson's Lemma that $K$ is a quotient of a sublattice of~$L$ and that~$L$ is a quotient of a sublattice of $K$. Therefore we can conclude $K\cong L$.
\end{proof}

\begin{corollary}
Let~$\cV$ and~$\cW$ be varieties of lattices such that every simple lattice in~$\cW$ contains a prime interval. If $\Conc\cV\subseteq\Conc\cW$ then either $\cV\subseteq\cW$ or $\cV\subseteq\dual{\cW}$.
\end{corollary}

Baker proves in \cite{B} that there exist continuum many varieties of locally finite modular lattices. By using products of projective planes of different characteristics, we obtain the corollary below.

\begin{corollary}
There are continuum many congruence classes of locally finite varieties of modular lattices.
\end{corollary}

By using \cite[Theorem 3.11]{G2}, which gives a lower bound for some critical points, we obtain the following result.

\begin{corollary}
Let~$\cV$ be a locally finite variety of modular lattices in which every simple lattice has length at most~$n$ and let $F$ be a field. If $\cV\not\subseteq\Var(\Sub F^n)$, then $\crit{\cV}{\Var(\Sub F^n)}=\aleph_2$.
\end{corollary}

There is an algorithm that, given finite lattices $K$, $L$ decides whether $\Var K\subseteq \Var L$. Therefore we can partially solve \cite[Problem~4]{GiWe1} for lattices and critical point~$\aleph_2$.

\begin{corollary}
There is an algorithm that, given finite lattices $K$, $L$, decides whether $\crit{\Var K}{\Var L}\le\aleph_2$ or  $\crit{\Var K}{\Var L}=\infty$.
\end{corollary}

\section{Functorial results}\label{S:functorresult}

In this section we study the existence of a functor $\Psi\colon\cV\to\cW$ between varieties of lattices such that $\Conc\circ\Psi\cong\Conc$. We prove that such  a functor exists if and only if either $\cV\subseteq\cW$ or $\cV\subseteq\dual{\cW}$. However, the functor itself does not need to be equivalent to either inclusion or dualization. The paper~\cite{GrWe} provides many examples of such functors, using the \emph{lattice tensor product}~$\boxtimes$. Let~$S$ be a simple bounded lattice, denote by $\cL$ the variety of all lattices. Then:
\begin{align*}
\Psi\colon\cL & \to\cL\\
A &\mapsto S\boxtimes A, & &\text{for each lattice~$A\in\cL$,}\\
f &\mapsto S\boxtimes f, & &\text{for each morphism of lattices $f$,}
\end{align*}
is a functor and $\Conc\circ\Psi$ is naturally equivalent to $\Conc$.

The following lemma grants the existence of congruence chains (cf. Definition~\ref{D:congruencechain} and Remark~\ref{R:congruencechain}).

\begin{lemma}\label{L:CC2whenFunctor}
Let~$A$ and~$B$ be lattices, let $f\colon A\to B$ be a morphism of lattices, and let $\pi_0,\pi_1\colon B\to A$ be morphisms of lattices. Assume that $\pi_0\circ f=\pi_1\circ f=\id_A$ and $\Conc B\cong \two^2$ with coatoms $\ker \pi_0$ and $\ker \pi_1$. There are $u<v$ in~$A$ and a congruence chain of~$B$ with extremities $f(u)$ and $f(v)$.
\end{lemma}

\begin{proof}
Let $u<v$ in~$A$. Set $\alpha_k=\ker\pi_k$, note that $\pi_k(f(u))=u<v=\pi_k(f(v))$, so $(f(u),f(v))\not\in\alpha_k$ for each $k<2$. It follows easily that $\Theta_B(f(u),f(v))=\one_B=\alpha_0\vee\alpha_1$.

There are an integer $n>0$ and a chain $f(u)=x_0<x_1<\dots x_n=f(v)$ of~$B$ such that $\Theta_B(x_i,x_{i+1})\in\set{\alpha_0,\alpha_1}$ for all $i<n$. Up to permuting~$\pi_0$ and~$\pi_1$, we can assume that $\Theta_B(x_0,x_1)=\alpha_1$, that is, $\pi_0(x_1)>\pi_0(x_0)=u$ and $\pi_1(x_1)=\pi_1(x_0)=u$.

Put $v'=\pi_0(x_1)$, put $t=x_1\wedge f(v')$. Hence $\pi_0(t)=\pi_0(x_1)\wedge\pi_0(f(v'))=v'\wedge v'=v'$ and $\pi_1(t)=\pi_1(x_1)\wedge\pi_1(f(v'))=u\wedge v'=u$, thus the following equalities hold:
\begin{align*}
u=\pi_0(f(u))=\pi_1(f(u))=\pi_1(t),\\
v'=\pi_0(f(v'))=\pi_1(f(v'))=\pi_0(t).
\end{align*}
Therefore $\Theta(f(u),t)=\alpha_1$ and $\Theta(t,f(v'))=\alpha_0$, that is $f(u)<t<f(v')$ is a congruence chain of~$B$.
\end{proof}

\begin{theorem}\label{T:nofunctor}
Let~$\cK$ be a class of bounded lattices containing every chain of length either~$1$, $2$, or~$3$, let~$\cW$ be a variety of lattices, and let $\Psi\colon\cK\to\cW$ be a functor such that $\Conc\circ\Psi\cong\Conc$. Then, up to changing~$\cW$ to its dual and composing~$\Psi$ with dualization, $\cW$ contains~$\cK$, and if $\Phi\colon\cK\to\cW$ denotes the inclusion functor, there exists a natural transformation $(\varepsilon_L)_{L\in\cK}\colon\Phi\to\Psi$ such that $\varepsilon_L$ is a congruence-preserving embedding from~$\Phi(L)$ into~$\Psi(L)$, for each $L\in\cK$.
\end{theorem}

\begin{proof}
Fix a natural equivalence $\vec \xi=(\xi_L)_{L\in\cK}\colon\Conc\circ\Psi\to\Conc$.

Let $\two=\set{0,1}$ be a chain. In this proof, each time we define a morphism $f_{P,Q}$ or $f_{P,Q}^k$ of lattices, we also denote $g_{P,Q}=\Psi(f_{P,Q})$ or $g_{P,Q}^k=\Psi(f_{P,Q}^k)$. Given any finite chain~$C$, we denote by $f_{\two,C}\colon\two\to C$ the only $0,1$-homomorphism of lattices.

Let~$C=\set{0,x_1,1}$ be a chain with $0<x_1<1$. Consider the following $0,1$-homomorphisms of lattices
\begin{align*}
f_{C,\two}^0\colon C&\to\two & f_{C,\two}^1\colon C&\to\two\\
x_1&\mapsto 0 & x_1&\mapsto 1.
\end{align*}
Notice that $\Conc\Psi(C)\cong\Conc C\cong \two^2$, $\set{\ker g_{C,\two}^0,\ker g_{C,\two}^1}=\At\Conc \Psi(C)$, and $f_{C,\two}^0\circ f_{\two,C}= f_{C,\two}^1\circ f_{\two,C}=\id_\two$. It follows from Lemma~\ref{L:CC2whenFunctor} that there are $u<v$ in $\Psi(\two)$ and a congruence chain $g_{\two,C}(u)=z_0^C<z_1^C<z_2^C=g_{\two,C}(v)$ of $\Psi(C)$. Up to changing~$\cW$ to its dual and dualizing $\Psi$, we can assume that $z_0^C<z_1^C<z_2^C$ is a direct congruence chain for $(\xi_C,C)$, that is,
\begin{equation}\label{eq:functor:dcc}
\xi_C(\Theta_{\Psi(C)}(z_0^C,z_1^C))=\Theta_C(0,x_1)\text{ and }\xi_C(\Theta_{\Psi(C)}(z_1^C,z_2^C))=\Theta_C(x_1,1).
\end{equation}
As $\ker f_{C,\two}^0=\Theta_C(0,x_1)$, $g_{C,\two}^0(z_1^C)=g_{C,\two}^0(z_0^C)$, similarly $g_{C,\two}^1(z_1^C)=g_{C,\two}^1(z_2^C)$. Therefore the following equalities hold:
\begin{align}
g_{C,\two}^0(z_1^C)=g_{C,\two}^0(z_0^C)=u=g_{C,\two}^1(z_0^C),\label{eq:functor:1}\\
g_{C,\two}^1(z_1^C)=g_{C,\two}^1(z_2^C)=v=g_{C,\two}^0(z_2^C).\label{eq:functor:2}
\end{align}
Given chains $D$ and $D'$ of the same length, we denote by $f_{D,D'}\colon D\to D'$ the only isomorphism.

Let $D$ be a chain of length~$2$, we denote $z_k^D=g_{C,D}(z_k^C)$ for all $k<3$. The following equalities hold:
\[
z_0^D=g_{C,D}(z_0^C)=g_{C,D}(g_{\two,C}(u))=g_{\two,D}(u).
\]
With a similar argument we obtain $z_2^D=g_{\two,D}(v)$. Moreover $g_{\two,D}(u)=z_0^D<z_1^D<z_2^D=g_{\two,D}(v)$ is a direct congruence chain of $\Psi(D)$.

Let $D=\set{0,y_1,y_2,1}$ be a chain with $0<y_1<y_2<1$. Consider the following $0,1$-homomorphisms of lattices
\begin{align*}
f_{C,D}^1\colon C&\to D & f_{C,D}^2\colon C&\to D\\
x_1&\mapsto y_1 & x_1&\mapsto y_2\\
f_{D,\two}^1\colon D&\to \two & f_{D,\two}^2\colon D&\to \two & f_{D,\two}^3\colon D&\to \two\\
y_1 &\mapsto 1 & y_1&\mapsto 0 & y_1&\mapsto 0\\
y_2 &\mapsto 1 & y_2&\mapsto 1 & y_2&\mapsto 0
\end{align*}
With a simple computation we obtain
\begin{align}
f_{C,\two}^0=f_{D,\two}^2\circ f_{C,D}^1=f_{D,\two}^3\circ f_{C,D}^1=f_{D,\two}^3\circ f_{C,D}^2\label{eq:functor:f1}\\
f_{C,\two}^1=f_{D,\two}^1\circ f_{C,D}^1=f_{D,\two}^1\circ f_{C,D}^2=f_{D,\two}^2\circ f_{C,D}^2\label{eq:functor:f2}
\end{align}

Put $z_0^D=g_{\two,D}(u)$, $z_1^D=g_{C,D}^1(z_1^C)$, $z_2^D=g_{C,D}^2(z_1^C)$, and $z_3^D=g_{\two,D}(v)$. As $f_{D,\two}^k\circ f_{\two,D}=\id_\two$, the following equalities hold:
\begin{align}\label{eq:functorgd2k1}
g_{D,\two}^k(z_0^D)=g_{D,\two}^k(g_{\two,D}(u))=u,\quad\text{for all $k\in\set{1,2,3}$}.
\end{align}
Similarly we obtain
\begin{align}\label{eq:functorgd2k2}
g_{D,\two}^k(z_3^D)=g_{D,\two}^k(g_{\two,D}(v))=v,\quad\text{for all $k\in\set{1,2,3}$}.
\end{align}

The following equalities hold:
\begin{align*}
g_{D,\two}^1(z_1^D)&=g_{D,\two}^1\circ g_{C,D}^1(z_1^C) &&\text{as $z_1^D=g_{C,D}^1(z_1^C)$.}\\
&=g_{C,\two}^1(z_1^C) &&\text{by~\eqref{eq:functor:f2}.}\\
&=v && \text{by~\eqref{eq:functor:2}.}
\end{align*}

With similar arguments~\eqref{eq:functor:1},~\eqref{eq:functor:2},~\eqref{eq:functor:f1},~\eqref{eq:functor:f2},~\eqref{eq:functorgd2k1}, and~\eqref{eq:functorgd2k2} implies the following equalities
\begin{align}
u=g_{D,\two}^1(z_0^D)=g_{D,\two}^2(z_0^D)=g_{D,\two}^3(z_0^D)= g_{D,\two}^2(z_1^D) = g_{D,\two}^3(z_1^D) = g_{D,\two}^3(z_2^D).\label{eq:functoru}\\
v=g_{D,\two}^1(z_3^D)=g_{D,\two}^2(z_3^D)=g_{D,\two}^3(z_3^D)= g_{D,\two}^1(z_1^D) = g_{D,\two}^1(z_2^D) = g_{D,\two}^2(z_2^D).\label{eq:functorv}
\end{align}

Notice that $g_{D,\two}^k(z_0^D) \le  g_{D,\two}^k(z_1^D)\le g_{D,\two}^k(z_2^D)\le g_{D,\two}^k(z_3^D)$ for all $k\in\set{1,2,3}$ and $z_i^D\not=z_j^D$ for $0\le i<j\le 3$. As $\bigcap\famm{\ker g_{D,\two}^k}{k\in\set{1,2,3}}=\zero_{\Psi(D)}$, this implies $z_0^D< z_1^D< z_2^D< z_3^D$.

The following equalities hold:
\begin{align}
\xi_D(\ker g_{D,\two}^1)&=\ker f_{D,\two}^1=\Theta_D(y_1,1),\label{eq:functor:cong1}\\
\xi_D(\ker g_{D,\two}^2)&=\ker f_{D,\two}^2=\Theta_D(0,y_1),\vee\Theta_D(y_2,1),\label{eq:functor:cong2}\\
\xi_D(\ker g_{D,\two}^3)&=\ker f_{D,\two}^3=\Theta_D(0,y_2).\label{eq:functor:cong3}
\end{align}
Therefore,
\begin{align*}
\xi_D(\Theta_{\psi(D)}(z_0^D,z_1^D))&\subseteq \xi_D(\ker g_{D,\two}^2)\cap \xi_D(\ker g_{D,\two}^3) && \text{by~\eqref{eq:functoru}}\\
&=(\Theta_D(0,y_1)\vee\Theta_D(y_2,1))\cap \Theta_D(0,y_2) &&\text{by~\eqref{eq:functor:cong2},~\eqref{eq:functor:cong3}}\\
&=\Theta_D(0,y_1)
\end{align*}
Similarly~\eqref{eq:functoru},~\eqref{eq:functorv},~\eqref{eq:functor:cong1},~\eqref{eq:functor:cong2}, and~\eqref{eq:functor:cong3} implies that:
\[
\xi_D(\Theta_{\psi(D)}(z_1^D,z_2^D))=\Theta_D(y_1,y_2),
\]
\[
\xi_D(\Theta_{\psi(D)}(z_2^D,z_3^D))=\Theta_D(y_2,1).
\]
Therefore, $g_{\two,D}(u)=z_0^D<z_1^D<z_2^D<z_3^D=g_{\two,D}(v)$ is a direct congruence chain of $\Psi(D)$.

We have constructed for each chain $D$ of~$\cK$ of length either~$2$ or~$3$ a direct congruence chain of $\Psi(D)$ with extremities $g_{\two,D}(u)$ and $g_{\two,D}(v)$. Let $L\in\cK$, let~$\cC$ be the set of all chains of~$L$ of length either~$2$ or~$3$. Let $\vec A=\famm{A_P,h_{P,Q}^L}{P\le Q\in\IC(\cC)}$ be the~$\cC$-chain diagram of~$L$. Put~$C_x=\set{0,x,1}$ for each $x\in L-\set{0,1}$. As $\vec \bB=\GA\circ\Psi\circ\vec A$ is a partial lifting of $\Conc\circ\vec A$ it follows from Theorem~\ref{T:good-liftingdia-imply-lifting} that the map $\varepsilon_L\colon  L\to B_\top=\Psi(L)$, defined by $\varepsilon_L(0)= \Psi(h_{\emptyset,\top}^L)(u)$, $\varepsilon_L(1)=\Psi(h_{\emptyset,\top}^L)(v)$, and $\varepsilon_L(x)=\Psi(h_{C_x,\top}^L)(z_1^{C_x})$ for all $x\in L-\set{0,1}$, is an embedding of lattices. Moreover $(\varepsilon_L,\xi_L)\colon(L,L,\Theta_L,\Conc L)\to\bB_\top=\GA(\Psi(L))$ is an embedding. As $\xi_L$ is an isomorphism, it follows that $\Conc\varepsilon_L\colon\Conc L\to \Conc(\Psi(L))$ is an isomorphism.

Let $p\colon K\to L$ be a morphism of lattices in~$\cK$. Let $x\in K-\set{0,1}$. Notice that the restriction of $p\colon C_x\to C_{p(x)}$ is the only isomorphism, thus $p\circ h_{C_x,\top}^K=h_{C_{p(x)},\top}^L\circ f_{C_x,C_{p(x)}}$. Therefore the following equalities hold:
\begin{align*}
\Psi(p)(\varepsilon_K(x))&=\Psi(p\circ h_{C_x,\top}^K)(z_1^{C_x})\\
&=\Psi(h_{C_{p(x)},\top}^L\circ f_{C_x,C_{p(x)}}) (z_1^{C_x})\\
&=\Psi(h_{C_{p(x)},\top}^L)( g_{C_x,C_{p(x)}}(z_1^{C_x}) )\\
&=\Psi(h_{C_{p(x)},\top}^L)(z_1^{C_{p(x)}})\\
&=\varepsilon_L(p(x)).
\end{align*}
Similarly $\Psi(p)(\varepsilon_K(0))=\varepsilon_L(p(0))$ and $\Psi(p)(\varepsilon_K(1))=\varepsilon_L(p(1))$. So $\vec\varepsilon$ is a natural transformation.
\end{proof}

\begin{remark}
Let~$C=\set{0,x_1,1}$ be a chain of length~$3$, let $f\colon \two\to C$ be the only $0,1$-homomorphism of lattices. Given an element~$a$ in a bounded lattice~$L$, denote by $f_{a,L}\colon C\to L$ the only $0,1$-homomorphism such that $f_{a,L}(x_1)=a$.

It follows from the proof of Theorem~\ref{T:nofunctor} that given $u<v$ in $\Psi(\two)$ and a direct congruence chain $\Psi(f)(u)<z_1<\Psi(f)(v)$ of $\Psi(C)$, we can construct $\vec\varepsilon\colon \Phi\to\Psi$ satisfying the conclusion of the theorem, and such that $\varepsilon_L(a)=\Psi(f_{a,L})(z_1)$, for each $L\in\cK$ and each $a\in L$.
\end{remark}

\begin{corollary}
Let~$\cV$ and~$\cW$ be varieties of lattices, let $\Psi\colon\cV\to\cW$ be a functor such that $\Conc\circ\Psi\cong\Conc$. Then either $\cV\subseteq\cW$ or $\cV\subseteq\dual{\cW}$.
\end{corollary}

\begin{corollary}\label{C:N5inM}
Let $\cN_5$ be the variety of lattices generated by the five-element non-modular lattice. Denote by $\cM$ the variety of all  modular lattices. Then there exists no functor $\Psi\colon\cN_5\to\cM$ such that $\Conc\circ\Psi\cong\Conc$.
\end{corollary}

\section{An open problem}\label{S:Pb}

The most natural problem brought up by the present work is whether the assumption about simple members of~$\cW$ is necessary for getting the result of Theorem~\ref{T:majorcritpoint}. It is not even known whether $\Conc\cV\subseteq\Conc\cW$ implies that~$\cV$ is contained either in~$\cW$ or its dual, for any varieties~$\cV$ and~$\cW$ of lattices. For example, it is not known whether $\Conc\cN_5\subseteq\Conc\cM$ (recall that~$\cN_5$ is the variety generated by the five-element non-modular lattice while~$\cM$ is the variety of all modular lattices). Nevertheless, Corollary~\ref{C:N5inM} suggests that this would not occur for any ``obvious'' reason.

More generally, even in crossover contexts (i.e., $\cV$ and~$\cW$ may not share the same similarity type), it is not completely unplausible that the chain diagram (cf. Definition~\ref{D:chaindia}) could be tailored to further varieties of algebras. If this could be done, then it might be possible to prove that for varieties of many other structures than just lattices, the containment $\Conc\cV\subseteq\Conc\cW$ could always be expressed \emph{via} a suitable amount of interpretability of members of~$\cV$ in~$\cW$. Hence, at least up to a suitable notion of interpretability, $\Conc\cV$ would always determine~$\cV$. This would be a most satisfactory answer to Question $\Qtwo$.

A partial answer to $\Qone$, for so-called \emph{strongly congruence-proper varieties} of algebras, is given in~\cite{G4} (note that a finitely generated congruence-modular varieties is strongly congruence-proper). If $\cV$ is a strongly congruence-proper variety of algebras, then $\Conc\cV$ is determined by its members of cardinality $\le\aleph_2$. The critical point between strongly congruence-proper varieties of algebras is either $\le\aleph_2$ or $\infty$.

\section{Acknowledgment}
I thank Friedrich Wehrung, for his many constructive remarks and his careful reading of the paper.


\begin{thebibliography}{99}

\bibitem{B}
K.\,A. Baker,
\emph{Equational classes of modular lattices}, Pacific J. Math.~\textbf{28}, no.~1 (1969), 9--15.

\bibitem{G1}
P. Gillibert, \emph{Critical points of pairs of varieties of algebras}, Internat. J. Algebra Comput.~\textbf{19}, no.~1 (2009), 1--40.

\bibitem{G2}
P. Gillibert,
\emph{Critical points between varieties generated by subspace lattices of vector spaces}, J. Pure Appl. Algebra~\textbf{214} (2010), 1306--1318.

\bibitem{G3}
P. Gillibert,
\emph{Categories of partial algebras for critical points between varieties of algebras}, Algebra Universalis, to appear, 59 pages.

\bibitem{G4}
P. Gillibert,
\emph{The possible values of critical points between strongly congruence-proper varieties of algebras}, Advances in Mathematics \textbf{257} (2014), 546--566.

\bibitem{GiWe1}
P. Gillibert and F. Wehrung,
\emph{{}From objects to diagrams for ranges of functors}, Lecture Notes in Mathematics, ~\textbf{2029}. Springer, Heidelberg, 2011. x+158 pp. ISBN: 978-3-642-21773-9.

\bibitem{GrWe}
G. Gr\"atzer and F. Wehrung,
\emph{A new lattice construction: the box product}, J. Algebra~\textbf{221}, no.~1 (1999), 315--344.

\bibitem{Ploscica00}
M. Plo\v s\v cica,
\emph{Separation properties in congruence lattices of lattices},
Colloq. Math. \textbf{83} (2000), 71--84.

\bibitem{Ploscica03}
M. Plo\v s\v cica,
\emph{Dual spaces of some congruence lattices},
Topology and its Applications \textbf{131} (2003), 1--14.

\bibitem{Ploscica04}
M. Plo\v s\v cica,
\emph{Separation in distributive congruence lattices},
Algebra Universalis \textbf{49}, no.~1 (2004), 1--12.

\bibitem{RTW}
P. R\r{u}\v{z}i\v{c}ka, J. T\r{u}ma, and F. Wehrung,
\emph{Distributive congruence lattices of congruence-permutable algebras},
J. Algebra \textbf{311}, no.~1 (2007), 96--116.

\bibitem{CLPSurv}
J. T\r{u}ma and F. Wehrung,
\emph{A survey of recent results on congruence lattices of lattices},
Algebra Universalis \textbf{48}, no.~4 (2002), 439--471.

\bibitem{NonMeas}
F. Wehrung,
\emph{Non-measurability properties of interpolation vector spaces},
Israel J. Math. \textbf{103} (1998), 177--206.


\end{thebibliography}
\end{document}